\documentclass[12pt,draft]{article}
\usepackage{amsmath,amsfonts,amsthm}
\usepackage{amssymb}
\usepackage{euscript}

\usepackage{color, soul}
\usepackage{tikz}
\usepackage{pgfplots}
\pgfplotsset{compat=1.5}
\usepackage{tikzsymbols}

\makeatletter
        \def\@thefnmark{\null}
        \def\footnotetexta{\@footnotetext}

\newtheorem{proposition}{Proposition}
\newtheorem{theorem}{Theorem}
\newtheorem{corollary}{Corollary}
\newtheorem{lemma}{Lemma}
\theoremstyle{definition}
\newtheorem{definition}{Definition}
\theoremstyle{definition}
\newtheorem{example}{Example}

\newcommand{\rhpolys}[2]{\EuScript{P}_r(^{#1} {#2})}

\newcommand{\npoly}[2]{\EuScript{P}(^{#1}{#2})}

\newcommand{\opoly}[2]{\EuScript{P\!}_o(^{#1}{#2})}
\newcommand{\opolyK}[1]{\EuScript{P}_o(^{#1}{C(K)})}

\newcommand{\normc}[1]{\|{#1}\|_d}
\newcommand{\normcdot}{\|\cdot\|_d}

\DeclareMathOperator{\diam}{diam}

\newcommand{\R}{\ensuremath{\mathbb{R}}}

\newcommand{\N}{\ensuremath{\mathbb{N}}}

\newcommand{\CK}{C(K)}
\newcommand{\MK}{\EuScript{M}(K)}

\title{	Geometry of Spaces of Orthogonally Additive Polynomials on
	$C(K)$}

\author{Christopher Boyd, Raymond A. Ryan and Nina Snigireva}

\date{}

\begin{document}

\maketitle

\begin{abstract}\footnotetexta{{\bf Keywords:} Orthogonally
		additive; Homogeneous polynomial; Banach Lattice; Regular Polynomial; Extreme point; Exposed point;
		Isometry.}
\footnotetexta{{\bf MSC(2010):} 46G25; 46G20; 46E10; 46B42; 46B04; 46E27.}
\noindent
We study the space of orthogonally additive $n$-homogeneous 
polynomials on $C(K)$.  There are two natural norms on this 
space.  First, there is the usual supremum norm of 
uniform convergence on the closed unit ball.  
As every orthogonally additive $n$-homogeneous polynomial
is regular with respect to the Banach lattice structure,
there is also the regular norm.  These norms are equivalent,
but have significantly different geometric properties.
We characterise the extreme points of the unit ball
for both norms, with different  results
for even and odd degrees. As an application,
we prove a Banach-Stone theorem. We conclude
with a classification  of the exposed points.
\end{abstract}

 
\section{Introduction}

A real function $f$ on a Banach lattice is 
said to be \emph{orthogonally additive}
if $f(x+y) = f(x)+f(y)$ whenever $x$ and  $y$ are disjoint.
Non-linear orthogonally additive functions on function spaces
often have useful integral representations --- see, for example
the papers of Chacon, Friedman and Katz \cite{Chacon,Friedman1,Friedman2}, Mizel \cite{Mizel} and 
Rao \cite{Rao}.  In 1990, Sundaresan \cite{Sund} initiated
the study of orthogonally additive $n$-homogeneous 
polynomials with particular reference to the spaces 
$L_p[0,1]$ and $\ell_p$ for $1\le p <\infty$.  Building on the work of Mizel,
he showed that, for every orthogonally additive $n$-homogeneous
polynomial $P$ on $L_p[0,1]$ with $n\le p$,
there exists
a unique function $\xi\in L_{\tilde{p}}$,
where $\tilde{p}=p/(p-n)$, such that
\begin{equation}\label{e:Sund}
P(x) = \int_0^1 \xi x^n\,d\mu
\end{equation}
for every $x\in L_p[0,1]$. When $n>p$, there are
no non-zero orthogonally additive $n$-homogeneous
polynomials on $L_p[0,1]$.  He went on to show
that the Banach space of orthogonally additive
$n$-homogeneous polynomials on $L_p[0,1]$
is isometrically isomorphic to $L_{\tilde{p}}$
where the latter space is equipped not with the 
usual norm, but with the equivalent norm
$\|x\|= 
\max\{\|x^+\|_{\tilde{p}},\|x^-\|_{\tilde{p}} \}$.
	
The next significant development was the discovery
of an integral representation for orthogonally additive
$n$-homogeneous polynomials on $C(K)$ spaces by 
P\'erez and Villanueva \cite{PV} and by Benyamini,
Lassalle and Llavona \cite{BLL}, who proved 
a representation of the form
\begin{equation}\label{e:IntRep}
P(x) = \int_K x^n\,d\mu
\end{equation}
where $\mu$ is a regular Borel signed measure on $K$.
The integral representations (\ref{e:Sund})
and (\ref{e:IntRep}) have been extended and 
generalized in various directions in recent years.
  See, for example, \cite{Palazuelos,Kusraeva11,Alaminos,Villena}.
  
  Orthogonally additive $n$-homogeneous polynomials
  are also of interest in the study of multilinear
  operators on Banach lattices and, more generally,
  on vector lattices.  If $E,F$ are vector lattices,
  an $n$-linear mapping $A\colon E^n \to F$ is 
  \emph{orthosymmetric} if $A(x_1,\dots,x_n) =0$
  whenever $x_i$ and $x_j$ are disjoint for some
  pair  of distinct indices $i,j$. Orthosymmetric
  multilinear mappings are automatically 
  symmetric \cite{Boul03}. In \cite{BuBuskes12}, Bu and Buskes prove   that an $n$-linear function
  is orthosymmetric if and only if the associate
  $n$-homogeneous polynomial is orthogonally additive.
  
  Let $E$ be a Banach lattice.  For every positive
  element $a$ of $E$, we may form the principal
  ideal
  $$
  E_a = \{x\in E: |x|\le na \text{ for some $n\in\N$} \}
  $$
  with lattice structure inherited from $E$ 
  and the norm defined by
  $\|x\|_a = \inf\{C>0:|x|\le Ca\}$.
  With this norm, $E_a$ is a Banach lattice.  By virtue of the Kakutani representation theorem \cite{KakAM},
  the Banach lattice $E_a$ is canonically 
  Banach lattice isometrically 
  isomorphic to  $C(K)$ for some compact Hausdorff topological
  space $K$, with $a$ being identified with the unit
  function on $K$.  The Banach lattice structure
  of $E$ is uniquely determined by its principal
  ideals.  It follows that an analysis of the
  orthogonally additive $n$-homogeneous
  polynomials on $C(K)$ is central to an
  understanding of the behaviour of
  orthogonally additive $n$-homogeneous 
  polynomials on general Banach lattices.
  
  In this paper, we focus on the geometric
  properties of the spaces $\opoly{n}{\CK}$
  of orthogonally additive $n$-homogeneous
  polynomials on $\CK$.  There are two 
  phenomena that are of particular interest.
  The first is that there are two natural
  ways to norm the space $\opoly{n}{E}$.
  The first is the norm of uniform convergence
  on the unit ball of $E$, given by
  $\|P\|_\infty= \sup\{|P(x)|: x\in E, \|x\|\le 1\}$.
  In this norm, $\opoly{n}{E}$ is a Banach space.
  Now $\opoly{n}{E}$ also has a lattice structure
  and so another choice of norm is the 
  \emph{regular norm}, defined by
  $\|P\|_r = \| |P| \|_\infty$, where $|P|$
  is the absolute value of $P$.  In this norm,
  $\opoly{n}{E}$ is a Banach lattice.
  The existence of these two
  norms was first observed by Bu and Buskes 
  \cite{BuBuskes12} and is hinted at in the
  paper of Sundaresan \cite{Sund}.  
  These  norms are equivalent, but
  we shall see that they
  have significantly different geometric properties.
  
  The second phenomenon is the influence of the 
  parity of the degree $n$ on the structure
  of the space $\opoly{n}{E}$ for the two norms.
  Bu and Buskes \cite{BuBuskes12} showed that, when $n$ is odd,
  the supremum and regular norms on $\opoly{n}{E}$
  are the same and
  that they are equivalent when $n$ is even.  We sharpen their results,
  using the strategy of working first on 
  $\opoly{n}{\CK}$ and then extending to 
  general Banach lattices.  
  The integral representation (\ref{e:IntRep})
  gives a canonical isomorphism between
  $\opoly{n}{\CK}$ and $\MK$, the space of 
  regular Borel signed measures on $K$.
  The regular norm on $\opoly{n}{\CK}$  corresponds to the 
  usual variation norm on $\MK$, but 
  the supremum norm is identified with 
  a different norm on $\MK$, given by
  $\|\mu\|_0 = 
  \max\{\|\mu^+\|_1,\|\mu^-\|_1\}$.
  We show that $(\MK, \|\cdot\|_0)$
  is isometrically isomorphic to 
  the dual space of $\CK$,
  where $\CK$ is endowed 
  with the norm $\normc{x}=
  \|x^+\|_\infty + \|x^-\|_\infty$ and we
  show that this norm is closely related
  to the \emph{diameter seminorm}
  (see, for example, \cite{Cabello}).  We use these identifications
  to give a complete description of the extreme
  points of the unit ball of $\opoly{n}{\CK}$
  for both norms, extending the results in \cite{CLZ1}. 
  Our starting point is a characterisation of the 
  extreme points in $\CK$ for the norm $\normcdot$
  and $\MK$ for the norm $\|\cdot\|_0$.
  This allows us to prove a  Banach-Stone
  theorem for $(\CK,\normcdot)$.  
  
  We finish with a study of the exposed points of
  the unit ball of the space $\opoly{n}{\CK}$.
  The identification of this space with the
  space of measures $\MK$, which is a dual space
  for both norms, allows us to use the
  theory of {\v S}mul'yan
   \cite{Smulyan1,Smulyan2}.  Using this machinery,
  we characterise the weak$^*$ exposed and the 
  weak$^*$ strongly exposed points of the unit ball.

\subsection*{Preliminaries}

Let $E$ be a real Banach space and let $n$ be
a natural number.  A function $P\colon E\to \R$ 
is an \emph{$n$-homogeneous polynomial} if there exists
a necessarily unique, bounded $n$-linear function
$A\colon E^n \to \R$ such that $P(x)= A(x,\dots,x)$
for all $x\in E$.  We write $P=\widehat{A}$ if $P$ and $A$
are related in this way.  The space $\npoly{n}{E}$
of $n$-homogeneous polynomials is a Banach space
with the supremum norm, 
$$
\|P\|_\infty= \sup\{|P(x)|:
x\in E, \|x\|\le 1 \} \,.
$$
The Banach space $\bigl(\npoly{n}{E},\|\cdot\|_\infty\bigr)$
is a dual space. We refer to the book by Dineen \cite{Dineen}
for this and other facts about $n$-homogeneous polynomials.

Now assume that $E$ is a Banach lattice. A partial order
is defined on $\npoly{n}{E}$ by $P=\widehat{A} \le
Q= \widehat{B}$ if $A(x_1,\dots,x_n) \le 
B(x_1,\dots,x_n)$ for all $x_1,\dots,x_n\ge 0$.
In particular, an $n$-homogeneous polymonial $P$
is said to be \emph{positive} if $P\ge 0$ in the
sense of this order and $P$ is \emph{regular} if it
is the difference of two positive $n$-homogeneous
polynomials.  The regular polynomials are precisely
those that have an \emph{absolute value}, which is
given by the formula
\begin{equation}\label{e:AbsFormula}
|P|(x) = \sup\Bigl\{\, \sum_{i_1,\dots,i_n}
|A(u^1_{i_1},\dots,u^n_{i_n})|: u^1,\dots,u^n \in \Pi(x)\Bigr\}\,,
\end{equation}
where $\Pi(x)$ denotes the set of partitions of $x$,
namely, all finite sets of positive elements of $E$
whose sum is $x$ \cite{BuBuskes12}.

The space $\rhpolys{n}{E}$ of regular $n$-homogeneous
polynomials on $E$ is a Banach lattice with the
\emph{regular norm},
$$
\|P\|_r = \| \,|P|\,\|_\infty\,.
$$
We have $\|P\|_\infty \le \|P\|_r$ and in general
these norms are not equivalent on $\rhpolys{n}{E}$.
Every regular $n$-homogeneous polynomial $P$ can
be decomposed canonically as the difference of
two positive $n$-homogeneous polynomials, so that 
$P=P^+ - P^-$ and $|P|= P^+ + P^-$.
We refer to  the paper 
of Bu and Buskes \cite{BuBuskes12} for further details.
For example, they show that $\bigl(\rhpolys{n}{E},\|\cdot\|_r\bigr)$  is a dual Banach lattice.

Let $K$ be a compact, Hausdorff space.  The space $\CK$
of continuous real functions on $K$ is a Banach
lattice with the supremum norm, 
$\|x\|_\infty= \sup\{|x(t)| : t\in K\}$.
We denote by $\MK$ the space of regular Borel signed
measures on $K$.  Then the Banach lattice dual
of $\CK$ can be identified with $\MK$ under the 
variation norm, which we denote by $\|\cdot\|_1$.
Thus, 
$$
\|\mu\|_1 = |\mu|(K) = \mu^+(K)+\mu^-(K)
=\|\mu^+\|_1 + \|\mu^-\|_1\,,
$$
 where $\mu^+$, $\mu^-$
are the positive and negative parts of $\mu$.


\section{Orthogonally additive $n$-homogeneous polynomials}

Let $E$ be a Banach lattice and $n$ a positive integer.
A  function  $P\colon E\to \R$
is called an \emph{orthogonally additive $n$-homogeneous
polynomial} if $P$ is a bounded $n$-homogeneous polynomial
with the property that $P(x+y)=P(x)+P(y)$
 whenever $x,y\in E$ are disjoint.  The space of
 orthogonally additive $n$-homogeneous polynomials
 on $E$ is denoted by $\opoly{n}{E}$.  It is easy to 
 see that $\opoly{n}{E}$ is a closed subspace
 of  the space $\bigl(\npoly{n}{E} ,\|\cdot\|_\infty\bigr)$ of bounded $n$-homogeneous
 polynomials with the supremum norm.
 Thus $\opoly{n}{E}$, with this norm, is a Banach
 space.  When $n=1$, this space is simply the dual
space $E'$, since every bounded linear functional
is orthogonally additive.
 
 We have the following integral representation
 for orthogonally additive $n$-homogeneous polynomials
 on $C(K)$ spaces, due to P\'erez-Garc\'ia and
 Villanueva \cite{PV} and Benyami, Lassalle and Llavona
 \cite{BLL} (see also \cite{CLZ1}).
 \begin{theorem}
 	\label{Rep}
 	Let $K$ be a compact, Hausdorff topological
 	space.  For every orthogonally additive 
 	$n$-homogeneous polynomial $P$ on $C(K)$
 	there is a regular Borel signed measure $\mu$ on
 	$K$ such that 
 	$$
 	P(x) = \int_K x^n\,d\mu
 	$$
 	for all $x\in C(K)$.
 \end{theorem}

In general, there is no guarantee that a Banach lattice
supports any non-trivial orthogonally additive polynomials of degree greater than one. Sundaresan
\cite{Sund} showed that there are no non-zero
orthogonally additive $n$-homogeneous polynomials
on $L_1[0,1]$ for $n>1$.  In the case of $\ell_1$,
it is easy to see that an $n$-homogeneous polynomial
$P$ is orthogonally additive if and only if there
exists a bounded sequence of real numbers, $(a_j)$,
such that 
$$
P(x) = \sum_{j=1}^\infty a_j x_j^n
$$
for every $x\in \ell_1$, and that $\|P\|_\infty = \sup_j |a_j|$.
Thus $\opoly{n}{\ell_1}$ is isometrically isomorphic
to $\ell_\infty$ for every $n$.

To put the results of the previous paragraph in a 
general context, we recall that a Banach lattice 
$E$ is an AL-space if the norm is additive on the 
positive cone: $\|x+y\| = \|x\| + \|y\|$ for
all $x,y\ge 0$.  The Kakutani representation theorem
\cite{KakAL,Lacey} states that 
every AL-space $E$ can be decomposed
into a disjoint sum of  copies of 
$\ell_1$ and $L_1$ spaces.
Accordingly,
$E$ is Banach lattice isometrically isomorphic to
a space of the form
$$
\Bigl[\ell_1(\Gamma) \oplus \bigl(
\bigoplus_{\alpha\in A } 
L_1[0,1]^{m_\alpha}\bigr )\Bigr]_1
$$
In this representation, the unit basis vectors $e_\gamma$
in $\ell_1(\Gamma)$ are in one-to-one correspondence
with the atoms in $E$ of unit norm. We recall that a positive 
element $x$ of $E$ is said to be an \emph{atom} if
$0\le y\le x$ implies that $y$ is a scalar multiple
of $x$.    We can write the 
second component in this representation as 
$L_1(\mu)$, where $\mu$ is the product of the Lebesgue
measures on the sets $[0,1]^{m_\alpha}$.  Thus, we
see that $E$ can be represented as the disjoint sum
$\ell_1(\Gamma)\oplus_1 L_1(\mu  )$, where the 
measure $\mu$ is  nonatomic. 

\begin{proposition}
	Let $E$ be an AL-space and let $n>1$.
	There is a non-zero orthogonally additive
	$n$-homogeneous polynomial on $E$ if and only
	if $E$ contains at least one atom.
\end{proposition}

\begin{proof}
	Let $\ell_1(\Gamma)\oplus_1 L_1(\mu  )$
	be the Kakutani representation of $E$
	as described above.
	
	Suppose that $E$ contains an atom.  Then the set $\Gamma$ in the Kakutani representation is non-empty.
	Choose $\gamma_0 \in \Gamma$ and define 
	$P(x) = x_{\gamma_0}^n$ for $x= (x_\gamma)\in
	\ell_1(\Gamma)$ and $P(x) = 0$ for $x\in L_1(\mu)$.
	Then $P$ is a non-zero orthogonally additive 
	$n$-homogeneous polynomial.
	
	Conversely, suppose that $P$ has no atoms.
	Then the Kakutani representation of $P$
	is $L_1(\mu)$ where the measure $\mu$ is 
	 nonatomic.  The proof in this case
	can be gleaned from \cite{Sund}, but we can
	 give a direct proof as follows.
	We treat the case $n=2$ for simplicity.
	Suppose that $P$ is an orthogonally additive
	$2$-homogeneous polynomial on $L_1(\mu)$,
	where $\mu$ is  nonatomic.  
	Let $A$ be the bounded, symmetric 
	bilinear form that generates $P$.
	Then $A$ is orthosymmetric:
	if $x,y$ are disjoint, then $A(x,y)=0$
	\cite[Lemma 4.1]{BuBuskes12}.
	It follows from the fact that 
	$L_1(\mu)\hat{\otimes}_\pi L_1(\mu)$
	is isometrically isomorphic to 
	$L_1(\mu^2)$ that there exists 
	$g\in L_\infty(\mu^2)$ such that
	$$
	A(x,y) = \int x(s)y(t) g(s,t)\,d\mu^2(s,t)
	$$
	If we take $x$, $y$ to be the characteristic functions
	of arbitrary disjoint measurable sets, this 
	integral is zero and so we have
	$$
	A(x,y) =
	\int_D x(t) y(t) g(t,t)\,d\mu^2\
	$$
	for all $x,y\in L_1(\mu)$, where $D$ is the diagonal.
	However,
	if $\mu$ has no atoms, then the product measure
	of the diagonal is zero.  Hence $P(x)=0$ for 
	every $x$.
\end{proof}

The Banach lattices
$L_1(\mu)$, where $\mu$ is  nonatomic, do not support
any real valued lattice homomorphisms.  Our next result
indicates that the existence of non-trivial 
$n$-homogeneous orthogonally additive polynomials
on a Banach lattice
is closely related to the existence of lattice homomorphisms.

\begin{proposition}
	Let $E$ be a Banach lattice, let $\varphi\in E'$ and let $n\ge 2$.  The $n$-homogeneous polynomial defined by $P(x)= \varphi(x)^n$
	is orthogonally additive if and only if either $\varphi$
	or $-\varphi$  is a lattice homomorphism.
\end{proposition}

\begin{proof}
	Suppose that $\varphi$ or $-\varphi$ is a lattice homomorphism. Then if $x$ and $y$ are disjoint, we have either $\varphi(x)$
	or $\varphi(y)=0$ and so $P(x+y)=P(x)+P(y)$.
	
	Conversely, suppose that $P=\varphi^n$ is orthogonally additive.  For every $x\in E$, the vectors $x^+$ and $tx^-$ are disjoint for all $t\in \R$. Therefore
	$$
	\varphi(x^+)^n+t^k\varphi(x^-)^n = P(x^+ + tx^-) =
	\sum_{j=0}^n \binom{n}{j} \varphi(x^+)^{n-j}\varphi(x^-)^j t^j\,.
	$$
	for every $t\in \R$.
	Hence either $\varphi(x^+)=0$ or $\varphi(x^-)=0$.  
	If we can show that $\varphi$ (or $-\varphi$) is positive,
	then it follows that $\varphi$ (or $-\varphi$) is a lattice homomorphism.
	
	Let $a$ be a positive element of $E$.  The principal 
	ideal $E_a$ generated by $a$ is isometrically Banach lattice
	isomorphic to $C(K)$ for some compact
	Hausdorff topological space $K$.  The functional $\varphi$ is represented by a regular Borel signed measure
	$\mu$ on $K$ and the fact that $\varphi(x^+)$ or
	$\varphi(x^-)=0$ for all $x\in E_a$ implies that the support of $\mu$
	consists of a single point.  It follows that either
	$\varphi$ or $-\varphi$ is positive on $E_a$.  Now 
	$E$ is the union of the principal ideals $E_a$, which
	are upwards directed by inclusion.  Thus, if $\varphi$
	(or $-\varphi$) is positive on one $E_a$, then $\varphi$
	(or $-\varphi$) is positive on all of $  E$.
\end{proof}

A Banach lattice $E$ is an AM-space if the norm
has the property that $x\wedge y=0$ implies
$\|x\vee y\| = \max\{\|x\|,\|y\|\}$.  In contrast with 
AL-spaces, there is a good supply of orthogonally
additive $n$-homogeneous polynomials on every AM-space.  
The 
Kakutani representation theorem for AM-spaces
 \cite{KakAM} shows
that the real valued lattice homomorphisms on an
AM-space $E$ separate
the points of $E$.  It follows that there is a rich
supply of orthogonally additive $n$-homogeneous
polynomials of every degree on $E$.

 \bigskip
 We now look at some properties of orthogonally additive
 polynomials on general Banach lattices.  Our starting
 point is the fact that every orthogonally additive
 $n$-homogeneous polynomial on a Banach lattice $E$ 
 is regular.  This has been shown by Toumi \cite[Theorem 1]{ToumiReg}.  One may also argue as follows.  Let $P$
 be an orthogonally additive $n$-homogeneous polynomial
 on a Banach lattice $E$.  As $E$ is the upwards
 directed union of its 
 principal ideals, it suffices to show that $P$ is
 regular on each of them.  Since each principal ideal
 is Banach lattice isometrically isomorphic to a
 $C(K)$, we can use the integral representation 
 in Theorem \ref{Rep}. Then the Jordan decomposition
 of the representing measure gives a decomposition of 
 the polynomial into the difference of two positive
  orthogonally
 additive $n$-homogeneous polynomials.
 Therefore $P$ is regular.
 
Let $P= \widehat{A}$ be a regular $n$-homogeneous polynomial $P$ on $E$. 
The absolute value of $P$ is given by
 \cite{BuBuskes12,Loane}
 \begin{equation}\label{e:AbsFormula2}
 |P|(x) = \sup\Bigl\{\, \sum_{i_1,\dots,i_n}
 |A(u^1_{i_1},\dots,u^n_{i_n})|: u^1,\dots,u^n \in \Pi(x)\Bigr\}
 \end{equation}
 for $x\ge 0$, where $\Pi(x)$ denotes the set 
 of partitions of $x$, namely, all finite sets
 of positive vectors whose sum is $x$.  In  general,
 we have
 \begin{equation}\label{e:Abs}
 |P(x)| \le |P|(|x|)
 \end{equation}
 for every $x\in E$ and $|P|$ is the smallest
 positive $n$-homogeneous polynomial, in the sense
 of the lattice structure of $\rhpolys{n}{E}$, with 
 this property.  The space $\rhpolys{n}{E}$
 is a Banach lattice in the \emph{regular norm},
 $$
 \|P\|_r = \| \,|P|\, \|_\infty\,.
 $$
 It follows from (\ref{e:Abs}) that $\|P\|_\infty \le \|P\|_r$
 for every $P\in \rhpolys{n}{E}$.
 In general, these norms are not equivalent.
 
 Now $\opoly{n}{E}$ is complete in the regular norm;
 indeed, it is even a dual Banach lattice 
 \cite[Theorem 5.4]{BuBuskes12}.  It follows that
 the supremum and regular norms are equivalent
 on this space.
 Thus, there is a sequence $(C_n)$ of 
 positive real numbers such that 
 $\|P\|_r \le C_n\, \|P\|_\infty$ for every
 $n$ and every $P\in \opoly{n}{E}$.
 Bu and Buskes \cite{BuBuskes12} show that 
 the two norms are the same for odd values of $n$.
 For even values of $n$,
 they show that $C_n \le n^n/n!$, the polarization
 constant.  We shall show that, in fact, $C_n=2$
 for even values of $n$ and that this is sharp.
 This will follow from estimates we give for the value
 of $|P|$ at positive points in $E$.
 
 If $\varphi$ is a bounded linear functional on $E$,
 then \cite{Meyer-Nieberg}
 $$
 |\varphi|(x) = \sup\{|\varphi(y)|: |y|\le x\}
 $$
 for every $x\ge 0$.  It would be suprising if there
 were such a simple formula for $|P|(x)$ when $P$
 is a regular $n$-homogeneous polynomial.  As a
 linear functional, $P$ acts on an $n$-fold symmetric
 tensor power of $E$ and the set of vectors $y$ 
 satisfying $|y|\le x$ is now a set of tensors,
 rather than elements of $E$.  However, if $P$
 is orthogonally additive, it is possible to 
 establish a relatively simple estimate for 
 the values of $|P|$.
 
  \begin{theorem} \label{p: basic}
 	Let $P$ be an orthogonally additive $n$-homogeneous polynomial on the  Banach lattice $E$.  
 	\begin{enumerate}
 		\item[(a)]
 		If $n$ is odd, then 
 		$$
 		|P|(x)   =
 		\sup\bigl\{|P(y)|: |y| \le x\bigr\}\,.
 		$$
 		for every $x\ge 0$ in $E$.
 		\item[(b)]
 		If  $n$ is even, then
 		$$
 		|P|(x)   \le 2\,
 		\sup\bigl\{|P(y)|:  |y| \le x\bigr\}\,.
 		$$
 		for every $x\ge 0$ in $E$.
 	\end{enumerate}
 \end{theorem}
 
 \begin{proof}
 	Let $x\ge 0$.  
 	It follows from  (\ref{e:AbsFormula}) that the value $|P|(x)$ is unchanged
 	if we consider   $P$ as an $n$-homogeneous polynomial on
 	the principal ideal $E_x$ generated by $x$.
 	Now $E_x$  is Banach lattice
 	isomorphic to $C(K)$ for some compact topological space $K$.
 	Since $P$ is orthogonally addive  there exists a regular signed Borel measure $\mu$ on $K$
 	such that 
 	$$
 	P(y)= \int_K y^n \,d\mu\,.
 	$$
 	for every $y\in E_x\cong C(K)$.
 	The symmetric $n$-linear form on $C(K)^n$ that
 	generates $P$ is given by
 	$$
 	A(x_1,\dots,x_n) = \int_K x_1\dots x_n\,d\mu\,.
 	$$
 	Thus, for $x_1,\dots,x_n \ge 0$,
 	$$
 	|A(x_1,\dots,x_n)| \le \int_K x_1\dots x_n\,d|\mu|
 	$$
 	and it follows that 
 	$$
 	|P|(x) \le \int_K x^n\,d|\mu|
 	$$
 	for $x\ge 0$.

 	Now in general, for a nonnegative function $w\in C(K)$ we have 
 	$$
 	\int_K w\,d|\mu| =
 	\sup\Bigl\{\Bigl| \int_K g\,d\mu\Bigr|: 
 	g \in C(K), |g| \le w \Bigr\}\,,
 	$$
 	Therefore
 	$$
 	|P|(x) \le   \sup 
 	\Bigl\{\Bigl|\int_K y\,d\mu\Bigr| :  y\in  C(K), \;|y| \le x^n \Bigr\}\,.
 	$$
 	where we are  identifying elements of $E$ with 
 	continuous functions on $K$.  We now consider separately the cases where $n$ is odd and even.
 	
 	(a)	We first consider the case when $n$ odd.\\
 	If $|y|\le x^n$, let $v= y^{1/n}$.
 	Then $|v|\le x$ and $\int_K y\,d\mu = \int_K v^n\,d\mu$.  Therefore
 	$$
 	|P|(x) \le   \sup 
 	\Bigl\{\Bigl|\int_K v^n\,d\mu\Bigr| :  v\in E, |v| \le x \Bigr\}\,.
 	$$
 	Thus we have
 	$$
 	|P|(x) \le \sup \{|P(y)| : |y| \le x\}\,.
 	$$
 	and it is easy to see that the reverse inequality also holds.
 	
 	(b) We now consider the case when $n$ even.  \\
 	We have
 	$$
 	|P|(x) \le  \sup 
 	\Bigl\{\Bigl|\int_K y\,d\mu\Bigr| :  |y| \le x^n \Bigr\}\,.
 	$$
 	Given $v\in E_x\cong C(K)$ satisfying $|v|\le x^n$,
 	we define $v_1,v_2\in C(K)$   by
 	$$
 	v_1(t)= \begin{cases} v(t)^{1/n} &\text{if } v(t)\ge 0\\
 	0 &\text{if } v(t) <0
 	\end{cases} \qquad 
 	v_2(t)= \begin{cases} 0 &\text{if } v(t)\ge 0\\
 	|v(t)|^{1/n} &\text{if } v(t) <0
 	\end{cases}
 	$$
 	Then $v= v_1^n - v_2^n$, and so
 	$$
 	\Bigl| \int_K v\,d\mu\Bigr|
 	\le \Bigl|\int_K v_1^n\,d\mu\Bigr| +
 	\Bigl| \int_K v_2^n\,d\mu\Bigl| =
 	\bigl| P(v_1)\bigr| +\bigl|P(v_2)\bigr|\,.
 	$$
 	It follows from $|v|\le x^n$ that $0\le v_1,v_2 \le x$.
 	Therefore
 	\begin{equation} \label{e:bound}
 	|P|(x)   \le 2\,
 	\sup\bigl\{|P(y)|:  0\le y  \le x\bigr\} =
 	2\,\sup\bigl\{|P(y)|:  |y| \le x\bigr\}\,,
 	\end{equation}
 	since $n$ is even.
 	
 \end{proof}
 
 To see that the bound in (\ref{e:bound}) for even values of $n$ is sharp,
 consider the example $P(x)=x_1^n-x_2^n$ on $\R^2$
 with any Banach lattice norm.  The bound is attained
 for the vector $x=(1,1)$.
 
 \begin{corollary}\label{c: basic}
 	Let $P$ is an orthogonally additive $n$-homogeneous
 	polynomial on a Banach lattice $E$. Then 
 	$\|P\|_r = \|P\|_\infty$ if $n$ is odd and
 	$\|P\|_\infty \le \|P\|_r \le 2\,\|P\|_\infty$ if $n$ is even.
 	These inequalities are sharp.
 \end{corollary}

 
\section{Orthogonally additive
	polynomials on $C(K)$}

In this section, we study the  supremum and regular 
norms on the 
spaces of orthogonally additive $n$-homogeneous
polynomials on $C(K)$.

The integral representation for orthogonally additive
$n$-homogeneous polynomials on $C(K)$ allows
us to identify the vector space $\rhpolys{n}{C(K)}$ with 
$\MK $, the space of regular Borel signed measures
on $K$.  The natural norm on $\MK \cong \CK' $ is the dual norm.  This is the 
 \emph{variation norm} for measures:
 $\|\mu\|_1 = |\mu|(K)$.  We shall see that this norm
 corresponds to the regular norm on the spaces of 
 orthogonally additive $n$-homogeneous polynomials.
 However, the supremum norm on $\opoly{n}{K}$ 
 corresponds to a different, but equivalent norm
 on the space of regular Borel signed measures.
 
 The space $\rhpolys{n}{C(K)}$ is a Banach lattice 
 with the regular norm, as is the dual Banach lattice
 $\MK $ with the variation norm.  We shall see that 
 the lattice structures of these two Banach lattices
 are the same.  We note that the lattice structure
 of $\MK $ as the dual of $C(K)$ is the same
 as the lattices structure of $\MK $ considered
 as a sublattice of the lattice of Borel signed
 measures on $K$. In other words, a measure 
 $\mu\in \MK $ is positive, in the sense that
 $\int_K f\,d\mu \ge 0$ for every nonnegative
 $x\in C(K)$, if and only if $\mu(E)\ge 0$
 for every Borel subset $E$ of $K$ \cite[Theorem 2.18]{Rudin}.
 
 \begin{proposition} \label{l:|P|}
 	Let $K$ be a compact Hausdorff topological space and let
 	$P$ be an orthogonally additive $n$-homogeneous
 	polynomial on $C(K)$, given by
 	$$
 	P(x) = \int_K x^n\,d\mu\,.
 	$$
 	Then the absolute value of $P$ is given by
 	$$
 	|P|(x) = \int_K x^n \,d|\mu|\,.
 	$$
 	
 \end{proposition}
 
 \begin{proof}
 	We have seen in the proof of Theorem~\ref{p: basic}
 	that
 	$$
 	|P|(x) \le \int_K x^n \,d|\mu|
 	$$
 	for every $x\ge 0$.
 	
 	To prove the reverse inequality, we start with the 
 	definition of the absolute value:
 	$$
 	|P|(x) = \sup\Bigl\{\, \sum_{i_1,\dots,i_n}
 	|A(u^1_{i_1},\dots,u^n_{i_n})|: u^1,\dots,u^n \in \Pi(x)\Bigr\}
 	$$
 	for $x\ge 0$
 	Taking each of $u^2,\dots,u^n$  to be the 
 	trivial partition $\{x\}$ gives
 	\begin{align*}
 	|P|(x) &\ge 
 	\sup\Bigl\{\, \sum_{i}
 	|A(u^1_{i},x,\dots,x)|: u^1 \in \Pi(x)\Bigr\}\\
 	& = 
 	\sup\Bigl\{\, \sum_{i}
 	\Bigl|\int_K u^1_i x^{n-1}\,d\mu\Bigr|: u^1 \in \Pi(x)\Bigr\} = \int_K x^n\,d|\mu|\,,
 	\end{align*}
 	applying the partition form of the Riesz-Kantorovich
 	formula for the absolute value of a linear
 	functional \cite[Theorem 1.16]{Aliprantis}
 	to the measure $d\lambda = x^{n-1}d\mu$ in $\MK $ and 
 	using 	the fact that $d|\lambda| = x^{n-1}d|\mu|$.

 \end{proof}

 	\begin{theorem}\label{p:isomorphism}
 		Let $K$ be a compact, Hausdorff space.
 	Let $J_n \colon \MK  \to \opoly{n}{C(K)}$
 	be given by
 	$$ 
 	(J_n\mu)(x) = \int_K x^n \,d\mu\,.
 	$$
 	\begin{itemize}
 		\item[(a)] For every $n$, $J_n$ is a
 		Banach lattice isometric
 		isomorphism from $\bigl(\MK ,\|\cdot\|_1\bigr)$ 
 		onto  $\bigl(\opoly{n}{C(K)},\|\cdot\|_r\bigr)$.
 		\item[(b)] If $n$ is odd, then the regular 
 		and supremum norms coincide on $\opoly{n}{C(K)}$
 		and so $J_n$ is an isometric isomorphism
 		for the supremum norm on $\opolyK{n}$.
 		\item[(c)] If $n$ is even, then $J_n$ is an
 		isometric isomorphism for  the norm on $\MK $
 		defined by
 		$$
 		\|\mu\|_0 := \max\{\|\mu^+\|_1,\|\mu^-\|_1\}
 		$$
		and 	the supremum norm
 		on $\opolyK{n}$.
 	\end{itemize}
 \end{theorem}
 
 \begin{proof}
 Cleary, $J_n$ is linear and surjective.  To see that it is 
 injective, suppose that the $n$-homogeneous 
 polynomial $P(x) = \int_K x^n \,d\mu$ is zero.
 The associated symmetric $n$-linear form
 is
 $$
 A(x_1,\dots,x_n) = \int_K x_1\dots x_n\,d\mu
 $$
and so $A(x_1,\dots,x_n)=0$ for all $x_1,\dots,x_n \in
C(K)$.  Taking $x_2= \dots = x_n = 1$, we have
$\int_K x\,d\mu=0$ for every $x\in C(K)$ and so $\mu=0$.

(a)  Clearly, $J_n$ is positive.  If we show that $J_n^{-1}$ 
is also positive, then it will follow that $J_n$ is a 
lattice homomorphism \cite[Theorem 7.3]{Aliprantis}.
Let $P=\widehat{A}$ be a positive element of $\rhpolys{n}{C(K)}$, with $\mu\in \MK $ satisfying
$J_n\mu = P$.  Then, for every nonnegative $x\in C(K)$,
we have $\int_K x\,d\mu = A(x,1,\dots,1) \ge 0$ and 
so $\mu$ is positive.  Therefore $J_n$ is a lattice
isomorphism for every $n$.

By Proposition \ref{l:|P|}, 
the regular norm of $P=J_n\mu$ is $\| P\|_r =
\|\,|P|\,\|_\infty = 
|P|(1) = |\mu|(K) = \|\mu\|_1$, since $|P|$ is 
increasing on the positive cone of $C(K)$.
Therefore $J_n$ is both a lattice isomorphism 
and an isometry.

(b) This has already been proved in Corollary \ref{c: basic}.

(c) Let  $\mu \in \MK $ and let $P= J_n\mu$.
It follows from (a) that $P^+= J_n \mu^+$
and $P^- = J_n \mu^-$. We have 
$$
P(x) = \int_K x^n \,d\mu^+ - \int_K x^n d\mu^-
$$
for every $x\in \CK$.  As $|a-b| \le \max\{|a|,|b|\}$
for $a,b\in \R^+$ and $n$ is even, it follows that 
$\|P\|_\infty \le \max\{\|\mu^+\|_1, \|\mu^-\|_1 \}$.

Now let $\{A,B\}$ be a Hahn decomposition of $\mu$,
with $\mu$ positive on $A$ and negative on $B$.
If $F \subset A$ is compact, then by a standard argument
using Urysohn's lemma 
(see, for example, \cite[Theorem 12.41]{Hewitt})
there is a decreasing sequence $(x_k)$ of continuous
functions on $K$ with values in $[0,1]$ that converges
almost everywhere with respect to $|\mu|$ to
$1_F$, the characteristic function of $F$.
Then, by the bounded convergence theorem,
$$
\|P\|_\infty \ge \lim_{k\to \infty}
\biggl| \int_K x_k^n\,d\mu\biggr|
= \biggl| \int_K 1_F \,d\mu \biggr|
= \mu^+(F)\,.
$$
It follows from the regularity of $\mu^+$
that $\|P\|_\infty \ge \mu^+(A) = \|\mu^+\|_1$. 
Similarly, $\|P\|_\infty \ge \|\mu^-\|_1$.
Therefore $\|P\|_\infty = \|\mu\|_0$.
\end{proof}

We summarize the identifications of the various norms,
bearing in mind that the supremum and regular norms
coincide for positive polynomials.
\begin{corollary}
	Let $P$ be an orthogonally additive $n$-homogeneous
	polynomial on $C(K)$, with corresponding
	measure $\mu\in \MK $. Then
	\begin{itemize}
		\item[(a)]  $\|P\|_r  = 
		\|P^+\|_r + \|P^-\|_r= \|\mu^+\|_1 + \|\mu^-\|_1 = \|\mu\|_1$.
		\item[(b)] If $n$ is odd, then $\|P\|_\infty = \|P\|_r$.
		\item[(c)] If $n$ is even, then  $\|P\|_\infty =
		\max\{ \|P^+\|_r, \|P^-\|_r \} =
		\max\{\|\mu^+\|_1, \|\mu^-\|_1\} = \|\mu\|_0$.
	\end{itemize}
\end{corollary}

 We note that the norm $\|\cdot\|_0$ is easily seen
 to be equivalent to the dual (variation) norm on
 $\MK $.  In fact, we have 
 $$ \|\mu\|_0 \le \|\mu\|_1 \le 2\,\|\mu\|_0
 $$ 
 for every
 $\mu\in \MK $.

It will be useful to have an alternative expression
for the norm $\|\cdot\|_0$ on $\MK$.  Using the identity
$\max\{a,b\} = \frac{1}{2} \bigl(|a+b| + |a-b|\bigr)$
for non-negative real numbers, we have
\begin{align*}
\|\mu\|_0 & = 
\frac{1}{2}\bigl( \bigl|\|\mu^+\|_1+
\| \mu^-\|_1\bigr| +
\bigl| \|\mu^+\|_1 - \|\mu^-\|_1  \bigr|
\bigr)   \\
 &= \frac{1}{2}\bigl(  \|\mu\|_1 + 
 \bigl| \mu^+(K)-\mu^-(K)\bigr| \bigr)
 = \frac{1}{2}\bigl( \|\mu\|_1 + \bigl| \mu(K)\bigl| \bigr)
 \label{e:mu0}
\end{align*}
 
 Thus, we have
 \begin{equation} \label{e:mu0}
 \|\mu\|_0 = \max\bigl\{\|\mu^+\|_1, \|\mu^-\|_1\bigr\}
 = \frac{1}{2}\Bigl( \|\mu\|_1 + \bigl| \mu(K)\bigl| \Bigr)
 \end{equation}
 These results clarify the geometric properties of the spaces $\opolyK{n}$; for the regular norm, these
 spaces are all essentially the same as the dual space
 $\MK $ with the variation norm.  
 The case of $\opolyK{n}$ with the supremum norm and $n$
 even is substantially different.  To understand this,
 we must study the extreme point structure of the
 unit ball of $\MK $ for the norm $\|\cdot \|_0$.

 
\section{Extreme points in  $\opoly{n}{C(K)}$}

In this section, we study the extreme points of the 
unit ball of the space $\opoly{n}{C(K)}$.  We begin
with the regular norm.  We have seen in Proposition
\ref{p:isomorphism} that there is an isometric
isomorphism
$$
\bigl(\opoly{n}{C(K)}, \| \cdot\|_r\bigr)
\cong \bigl(\MK , \|\cdot\|_1\bigr)
$$
where $\|\cdot\|_1$ denotes the variation norm
on $\MK $, the space of regular Borel signed 
measures on $K$.  Furthermore, when the degree
$n$ is odd, the supremum and regular norms
on $\opoly{n}{C(K)}$ coincide.

It is a classical result that the extreme
points of the unit ball of $\MK $ for 
the variation norm are the measures of the 
form $\pm \delta_t$, where $t\in K$ 
(see, for example, \cite[V.8.6]{DS}).
The isomorphism between $\opoly{n}{C(K)}$
and $\MK $ associates the polynomial
$P(x)= x(t)^n$ with the measure $\delta_t$.
Thus, we have

\begin{proposition} \label{p:extremeRegular}
	Let $K$ be a compact Hausdorff topological space.
	The extreme points of the closed unit
	ball of the space 
	$\bigl(\opoly{n}{C(K)},\|\cdot\|_r\bigr)$ are 
	the $n$-homogeneous polynomials $\pm \delta_t^n$,
	where $t\in K$ and $\delta_t^n(x)= x(t)^n$.
\end{proposition}

This result is given in \cite{CLZ1} for 
the supremum norm, but the proof given there
is not valid for polynomials of even degree.  
However, this does not affect the results that 
follow in \cite{CLZ1}. 
In particular, their elegant proof 
of the integral representation still stands. 
Essentially, all that is required
for
their arguments to work is that $\opolyK{n}$ is a dual space and that the extreme points of the unit ball
are as described above.

We now turn to the geometry of $\opolyK{n}$
for the supremum norm, where the degree $n$
is even.  We have the isometric
isomorphism
\begin{equation*} \label{e: iso}
\bigl(\opoly{n}{C(K)}, \|\cdot\|_\infty\bigr)
\cong \bigl(\MK , \|\cdot\|_0\bigr)
\end{equation*}
where $\|\mu\|_0 = \max\{\|\mu^+\|_1, \|\mu^-\|_1\}$.
We will show that $\|\cdot\|_0$ is the dual 
of a norm on $\CK$ that is equivalent to 
the supremum norm.

The norm we seek is related to the 
\emph{diameter seminorm} on $\CK$, which is defined by
$$
\rho(x) = \diam(x) = \sup\{|x(s)-x(t)|: s,t\in K\}\,.
$$
It is easy to see that we also have
$$
\rho(x) = 2\,\inf\bigl\{\|x-\alpha 1_K\|_\infty: 
\alpha\in \R\bigr\}
$$
The kernel of $\rho$ is the one dimensional 
subspace
 of constant functions.
 As in \cite{Cabello}, we use 
 $C_\rho(K)$ to denote the quotient
 space $\CK/ \ker\rho$. It is a Banach
 space under the norm
 $$
 \|\pi(x)\|_\rho = \rho(x)
 $$
 where $\pi\colon \CK \to \CK/\ker \rho$
 is the quotient map.
 Following Cabello-Sanchez \cite{Cabello}, we note
 that this means that
  $\bigl(C_\rho(K),\|\cdot\|_\rho\bigr)$
  is isometrically isomorphic, up to a a constant
  factor $2$, to the quotient space of
  $\bigl(\CK, \|\cdot\|_\infty\bigr)$
  by the subspace of constant functions.
  Therefore the dual space $\bigl(C_\rho(K),\|\cdot\|_\rho\bigr)'$ 
  is isometrically isomorphic, up to a 
  constant factor $1/2$, to a subspace
  of $\bigl(\CK, \|\cdot\|_\infty\bigr)'$,
  the space of regular Borel signed measures
  with the variation norm.
  This subspace is the space of measures $\mu$
  satisfying $\mu(K)=0$ and on it we have \cite{Cabello}
  $$
  \|\mu\|_1 = 2\|\mu\|_{\bigl(C_\rho(K),\|\cdot\|_\rho\bigr)'}
  \,.
  $$
  \begin{theorem}[Cabello-Sanchez \cite{Cabello}]
  	\label{p:Cabello}
  	Let $K$ be a compact Hausdorff topological space.
  	A regular Borel signed measure $\mu$ is an 
  	extreme point of the unit ball of the dual
  	space 
  	$\bigl(C_\rho(K),\|\cdot\|_\rho\bigr)'$
  	if and only if 
  	$\mu = \delta_s - \delta_t$,
  	where $s$ and $t$ are distinct points of $K$.
  \end{theorem}

In order to apply this result, we first need to 
identify the predual of the norm $\|\cdot\|_1$
 on $\MK \cong \CK'$.
 \begin{theorem}\label{p: predual}
 	Let $K$ be a compact Hausdorff topological space. 
 	Let $\normcdot$
 	be the norm on $C(K)$ defined by
 	\begin{equation}\label{e:predual}
 	\normc{x} := \|x^+\|_\infty + \|x^-\|_\infty 
 	= \max\bigl\{\|x\|_\infty,\rho(x) \bigr\}
 	\end{equation}
 	where $\rho$ is the diameter seminorm.
 	Then the dual space of 
 	$\bigl(C(K), \normcdot\bigr)$
 	is isometrically isomorphic to the space
 	of regular Borel signed measures on $K$
 	with the norm $\|\mu\|_0
 	= \max\{\|\mu^+\|_1, \|\mu^-\|_1\}$.
 	
 \end{theorem}
\begin{proof}
	A routine calculation shows that  the formula 
	$\normc{x} = \|x^+\|_\infty + \|x^-\|_\infty$
	defines a norm on $\CK$.  To establish the second
	equality in (\ref{e:predual}), we consider two cases.
	\begin{enumerate}
		\item[(a)] Suppose the function $x$ has constant
		sign. Then $\rho(x) \le \|x\|_\infty$  
		and one of $\|x^+\|_\infty$, 
		$\|x^-\|_\infty$ is zero. Therefore $\normc{x}
		=\|x\|_\infty$.
		\item[(b)] If $x$ changes sign, then
		$\|x\|_\infty \le \|x^+\|_\infty + \|x^-\|_\infty = \rho(x)$. 
	\end{enumerate}
Therefore $\normc{x}  	= 
\max\bigl\{\|x\|_\infty,\rho(x) \bigr\}$
for every $x\in \CK$.

Let us denote  the dual norm of $\|\cdot\|_1$ by $\|\cdot\|_1'$.  If $x\in \CK$ and $\mu\in\MK$,
then
$$
\int_K x\,d\mu = 
\Bigl( \int_K x^+\,d\mu^+ + \int_K x^-\,d\mu^-\Bigr)
-\Bigl(\int_K x^+\,d\mu^- + \int_K x^-\,d\mu^+\Bigr)\,.
$$
Now
$$
0\le \int_K x^+\,d\mu^+ + \int_K x^-\,d\mu^-
\le \|x^+\|_\infty \|\mu^+\|_1 + \|x^-\|_\infty \|\mu^-\|_1
\le \normc{x} \|\mu\|_0
$$
and similarly
$$
0\le \int_K x^+\,d\mu^- + \int_K x^-\,d\mu^+
\le \normc{x} \|\mu\|_0 \,.
$$
Therefore
$$
\Bigl| \int_K x\,d\mu \Bigr| \le \normc{x} \|\mu\|_0 
$$
and so  $\|\mu\|_1' \le \|\mu\|_0$.

Fix $\mu\in \MK$ and let $\varepsilon>0$.
Let $\{A,B\}$ be a Hahn decomposition for $\mu$,
where $A$ is a positive set and $B$ a negative set.
Since $\mu$ is regular, there exist compact sets
$C\subset A$ and $D\subset B$ such that
$|\mu|(A\setminus C),  |\mu|(B\setminus D) 
<\varepsilon$.
By Urysohn's lemma, there is a continuous function
$y\colon K \to [0,1]$ that takes the values $1$ and $0$
on the sets $C$ and $D$ respectively. 
Then $\|y\|_1 = 1$ and 
$$
\int_K y\,d\mu = \int_C y\,d\mu
+ \int_{A\setminus C} y\,d\mu
+ \int_{B\setminus D} y\,d\mu
$$
It follows that 
$$
\Bigl|\int_K y\,d\mu\Bigl| \ge
\mu^+(C)-2\varepsilon 
\ge \mu^+(A) -3\varepsilon = \|\mu^+\|_1 -3\varepsilon\,.
$$
Similarly, 
$$
\Bigl|\int_K y\,d\mu\Bigr| \ge \|\mu^-\|_1 - 3\varepsilon
$$
Thus, $\|\mu\|_1' \ge \|\mu\|_0 - 3\varepsilon$
for every $\varepsilon>0$

Therefore $\|\mu\|_1' = \|\mu\|_0$
for every $\mu\in \MK$.
\end{proof}

\subsection{The extreme points of 
the unit ball of $\bigl(\CK,\normcdot\bigr)$}

The extreme points of the closed unit ball of 
$\CK$ with the supremum norm are the 
constant functions $\pm 1$.  Our next result
shows that changing to the equivalent norm
given in the preceding proposition 
leads to a different set of extreme points.

\begin{theorem}
	A function $x$ is an extreme point of the closed unit
	ball of $(C(K),\normcdot)$ if and only if
	either\\
	(i) $x(t)= 1$ or $0$ for every $t\in K$,
	or\\
	(ii) $x(t)= -1$ or $0$ for every $t\in K$\\ 
	(and $\{t: x(t)\neq 0 \}\neq \varnothing$ in each case.)
\end{theorem}

\begin{proof}
	To show that every such function is extreme,
	let $\normc{x}=1$, with $x(t)=1$ for 
	$t\in A$ and $x(t)=0$ for $t\in A^c$, where
	$A$ is a nonempty subset of $K$.  Suppose that
	$$
	x = a y + bz\,,
	$$
	where $a,b\in (0,1)$ with
	$a+b=1$ and $\normc{y} = \normc{z} = 1$.
	Then, for $t\in A$, $ay(t)+bz(t) =1$.
	But $|y(t)|, |z(t)| \le  1$ and it follows that
	$y(t)= z(t)=1$ for every $t\in A$.

	Now, if $t\in A^c$, then $ay(t)+bz(t) = 0$.
	But $\diam(y), \diam(z) \le 1$ and $\|y\|_\infty,
	\|z\|_\infty =1$ imply that $0\le y(t), z(t)
	\le 1$ for every $t\in K$ and hence
	$y(t)= z(t) = 0$ for every $t\in A^c$.
	Therefore $y(t)=z(t)=x(t)$ for every $t\in K$
	and so $x$ is an extreme point.
	The case in which $x$ takes values $-1$ and $0$
	is done in exactly the same way.
	
	We now show that every extreme point is of this type.
	Let $x$ be an extreme point.  Since
	$\normc{x} = 
	\max\bigl\{\|x\|_\infty, \diam(x) \bigr\} =1$,
	there are two cases to consider.
	
	Case 1:  $\|x\|_\infty =1$ and $\diam(x) \le 1$.
	Then $x$ takes its values either in $[-1,0]$
	or $[0,1]$.  Suppose it is the latter.  Then there
	is at least one point at which $x(t)=1$.  Suppose
	there is a point $s\in K$ for which $0<x(s) <1$.
	Then, by a standard argument, there is a function
	$y\in C(K)$ with values in $[0,1]$ and supported
	by a neighbourhood of $s$, such that
	$\|x\pm y\|_\infty \le 1$. Clearly, we also have
	$\diam(x\pm y)\le 1$.  This implies that $x$ is 
	not extreme and so we have a contradiction.
	Therefore $x$ can only have values $0$ or $1$.

	Case 2: $\diam(x) =1$ and $\|x\|_\infty < 1$.
	There exist points $s,t$ in $K$ such that 
	$|x(t)-x(s)|= \diam(x)=1$.  Without loss of generality
	we may assume that $x(t)>x(s)$. Then $x$ takes its
	values in the interval $[x(s),x(t)]$.
	If there exists $u\in K$ such that $x(s)< x(u)<x(t)$,
	then, using the same perturbation argument as in 
	the proof of Case 1, it follows that $x$ is not extreme.
	Therefore $x$ has precisely two distinct values,
	$x(s)$ and $x(t)$.
	
	Suppose that $-1<x(s)<x(t)<1$.  Then, for sufficiently
	small $\varepsilon>0$, we have
	$\normc{x\pm \varepsilon 1_K} = 1$, which implies
	that $x$ is not extreme.  
	Therefore either $x(t)=1$ and $x(s)=0$,
	or $x(s)=-1$ and $x(t)=0$.
	
\end{proof}

Note that if $K$ is connected, then the closed unit
ball of $C(K)$ for both the supremum norm
and the norm $\normcdot$ has the same extreme 
points --- the constant functions $1$ and $-1$.
However, if $K$ has more than one connected component, then
there are functions that are extreme for $\|\cdot\|_\infty$
but not $\normcdot$, and vice versa.

\subsection{The extreme points of the unit ball of
 $\bigl(\MK, \|\cdot\|_0\bigr)$}
The isometric isomorphism
$$
\bigl(C(K),\normcdot)' \cong
\bigl(\MK, \|\cdot\|_0\bigr)\,.
$$
now enables us to  identify the extreme points of the unit
ball of $\bigl(\MK,\|\cdot\|_0\bigr)$.

\begin{theorem} \label{p:extremeMK}
	Let $K$ be a compact Hausdorff topological space. A regular
	Borel signed measure $\mu$ on $K$ 
	is an extreme point of the unit ball of 
	$\bigl(\MK,\|\cdot\|_0\bigr)$ if and only
	if it is one of the following:
	\begin{itemize}
		\item[(a)] $\mu = \pm \,\delta_t$, where 
		$t\in K$;
		\item[(b)] 
		$\mu = \delta_s-\delta_t$, where $s$, $t$
		are distinct points in $K$.
	\end{itemize}
\end{theorem}
\begin{proof}
	\mbox{}\\
	Step 1.  We show that every extreme point must 
	be 	one of the types described in the statement.
Let $\tilde{K}$ be the space $K \cup K^2$, with 
the sum topology, where $K^2$ carries the product
topology.  For $x\in \CK$, let
$\tilde{x}$ be the continuous function on $\tilde{K}$
defined by $\tilde{x}(u) = x(u)$ for $u\in K$
and $\tilde{x}(s,t) = x(s)-x(t)$ for $(s,t)\in K^2$.
The fact that 
$$
\|\tilde{x}\|_\infty =
\max\bigl\{\sup\{|x(u)|:u\in K\},
\sup\{|x(s,t)| : s,t\in K \}\bigr\} = \normc{x}
$$
shows that the mapping $x\mapsto \tilde{x}$
is an isometric embedding of $\bigl(C(K),\normcdot\bigr)$
into a closed subspace of 
$\bigl(C(\tilde{K}),\|\cdot\|_\infty \bigr) $.  
It follows from \cite[V.8.6]{DS} that every extreme
point of the unit ball of 
$\bigl(C(K),\normcdot)' \cong
\bigl(\MK,\|\cdot\|_0\bigr)$
is either $\pm\delta_u$ for some $u\in K$,
or $\delta_s-\delta_t$ for some $s,t\in K$.

\medskip
	Step 2.  
	$\pm\,\delta_t$ are extreme points: 
	Suppose that 
	$$
	\delta_t = a \mu_1 + b\mu_2 \,,
	$$
	where $\mu_1, \mu_2\in \MK$, 
	$\|\mu_1\|_0 = \|\mu_2\|_0=1$, $a,b\in (0,1)$
	and $a+b=1$.
	Applying $\delta_t$ to the function $1_K$, we have
	$$
	a\mu_1(K) + b\mu_2(K) = 1\,.
	$$
	On the other hand, $\normc{1_K} = 1$ implies that
	$|\mu_i(K)| \le 1$ for $i=1,2$.
	Therefore $\mu_1(K)= \mu_2(K) = 1$ and it 
	follows from $\|\mu_i\|_0 = 
	\frac{1}{2} \bigl(\|\mu_i\|_1+ |\mu_i(K)|\bigr)
	=1$ that $\|\mu_1\|_1 = \|\mu_2\|_1 = 1$.
	Since $\delta_t$ is an extreme point of 
	the unit ball of $\MK$ for the variation norm,
	it follows that $\mu_1=\mu_2 = \delta_t$.
	Therefore $\delta_t$ is an extreme
	point of the unit ball of
	$\bigl(\MK,\|\cdot\|_0\bigr)$.

	\medskip
		Step 3. 
		$\delta_s-\delta_t$ is extreme for every pair
		of distinct points 
		$s,t\in K$:
		Suppose that 
		$$
		\delta_s - \delta_t = a \mu_1 + b\mu_2 \,,
		$$
		where $\mu_1, \mu_2\in \MK$, 
		$\|\mu_1\|_0 = \|\mu_2\|_0=1$, $a,b\in (0,1)$
		and $a+b=1$.
		Without loss of generality, we may assume
		that $a=b= \frac{1}{2}$.
		As $\|\delta_s - \delta_t\|_1=2$, we have
		$4 \le \|\mu_1\|_1 + \|\mu_2\|_1$.
		On the other hand,
		$$
		\|\mu_i\|_0 = \frac{1}{2}\bigl( \|\mu_i\|_1 +
		|\mu_i(K)| \bigr) = 1\quad\text{for $i=1,2$}
		$$
		and it follows that $\|\mu_i\|_1 =2$ and
		$\mu_i(K) = 0$ for $i=1,2$.
		Therefore $\delta_s - \delta_t$, $\mu_1$ and $\mu_2$
		all lie in $\bigl(C_\rho(K),\|\cdot\|_\rho\bigr)'$,
		the space of regular Borel signed measures on $K$
		that are zero on $K$.  Furthermore, these measures
		are all unit vectors in this space, since the variation
		norm is exactly twice the dual norm in $\bigl(C_\rho(K),\|\cdot\|_\rho\bigr)'$.
		It follows from the result of Cabello-Sanchez 
		(Theorem \ref{p:Cabello} above) that 
		$\mu_1 = \mu_2 = \delta_s - \delta_t$.
		Therefore $\delta_s - \delta_t$ is an extreme
		point of the unit ball of
		$\bigl(\MK,\|\cdot\|_0\bigr)$.
		 
\end{proof}

We can now describe the extreme points of the unit
ball of $\opolyK{n}$ for the supremum norm.
Recall that, when $n$ is odd, the supremum and 
regular norms coincide.  Thus, by Propositions
\ref{p:isomorphism},
\ref{p:extremeRegular} and Theorem \ref{p:extremeMK}
we have the following result.

\begin{corollary} \label{p:extremeSup}
	Let $K$ be a compact, Hausdorff space.
	\begin{itemize}
		\item[(a)]
		If $n$ is odd, then   $P\in \opolyK{n}$
		is an extreme point of the closed unit
		ball of the space 
		$\bigl(\opoly{n}{C(K)},\|\cdot\|_\infty)$ 
		if and only if $P= \pm \delta_t^n$,
		for some $t\in K$,
		where
		$$ \delta_t^n(x)= x(t)^n\,.
		$$
		\item[(b)] 
		If $n$ is even, then $P\in \opolyK{n}$ is an extreme 
		point of the closed unit
		ball of the space 
		$\bigl(\opoly{n}{C(K)},\|\cdot\|_\infty)$ 
		if and only if $P$ is one of the following:
		\begin{itemize}
		\item[(i)] 
		$P = \pm \delta_t^n$ for some $t\in K$, 
		where $\delta_t^n(x)= x(t)^n$;
		\item[(ii)] 
		$P= \delta_s^n - \delta_t^n$, where $s,t$ are distinct points
		in $K$, and 
		$$
		(\delta_s^n - \delta_t^n)(x) = x(s)^n - x(t)^n \,.
		$$
		\end{itemize}
	\end{itemize}

\end{corollary}


\begin{example}
	Suppose that the compact, Hausdorff space $K$
	has just two points, $\alpha, \beta$.  Then
	the vector lattice $\CK$ can be identified 	with $\R^2$, where $(x_1,x_2)\in \R^2$
	corresponds to the function 
	$\alpha \mapsto x_1$, $\beta\mapsto x_2$.
	The supremum norm on $\CK$ is identified
	with the supremum norm on $\R^2$.
	The orthogonally additive $n$-homogeneous
	polynomials on $\R^2$ have the form
	$P(x) = a_1 x_1^n + a_2 x_2^n$. The regular and supremum
	norms are
	\begin{align*}
	\|P\|_r &= |a_1| + |a_2| \,,\\
	\|P\|_\infty & = 
	\begin{cases}
		|a_1| + |a_2|\,, \quad\text{if $n$ is odd,}\\
		\max\{|a_1|, |a_2|, |a_1+a_2|  \} \,, 
		\quad\text{if $n$ is even.}
	\end{cases}
	\end{align*}
	The diagrams below show the unit balls
	for both norms.
\end{example}



\begin{minipage}{0.46\linewidth}
	\begin{center}
		\begin{tikzpicture}[xscale=1,yscale=0.9,domain=-1.6:1.6]
		\begin{axis}[
		width=5.5cm,
		height=5cm,
		axis lines = middle, enlargelimits = true,
		]
		\addplot[red,very thick,domain=0:1] plot {-x+1}; 
		\addplot[blue,mark=*]  coordinates {(1,0)};
		\addplot[blue]  coordinates {(1.2,0.2)} node{\large{$x_1^n$}};
		\addplot[red,very thick,domain=-1:0] plot {x+1}; 
		\addplot[blue]  coordinates {(0.3,1)} node{\large{$x_2^n$}};
		\addplot[blue,mark=*]  coordinates {(0,1)};
		\addplot[red,very thick,domain=-1:0] plot {-x-1}; 
		\addplot[blue,mark=*]  coordinates {(0,-1)};
		\addplot[blue]  coordinates {(0.5,-1)} node{\large{$-x_2^n$}};
		\addplot[red,very thick,domain=0:1] plot {x-1}; 
		\addplot[blue,mark=*]  coordinates {(-1,0)};
		\addplot[blue]  coordinates {(-1.4,0.2)} node{\large{$-x_1^n$}};
		\end{axis}
		\end{tikzpicture}
	\end{center}
	
\end{minipage}
\hfill
\begin{minipage}{0.46\linewidth}
	\begin{center}
		\begin{tikzpicture}[xscale=1,yscale=1.15,domain=-3:3]
		\begin{axis}[
		width=5.5cm,
		height=5cm,
		axis lines = middle, enlargelimits = true,
		]
		\addplot[red,very thick,domain=0:1] plot {-x+1}; 
		\addplot[blue,mark=*]  coordinates {(1,0)};
		\addplot[blue]  coordinates {(1.2,0.3)} node{{$x_1^n$}};
		\addplot[red,dashed, very thick,domain=-1:0] plot {x+1}; 
		\addplot[blue,mark=*]  coordinates {(0,1)};
		\addplot[blue]  coordinates {(0.3,1.3)} node{{$x_2^n$}};
		\addplot[red,very thick,domain=-1:0] plot {-x-1}; 
		\addplot[blue,mark=*]  coordinates {(0,-1)};
		\addplot[blue]  coordinates {(-0.4,-1.4)} node{{$-x_2^n$}};
		\addplot[red,dashed, very thick,domain=0:1] plot {x-1}; 
		\addplot[blue,mark=*]  coordinates {(-1,0)};
		\addplot[red, very thick,domain=-1:0] plot {1} ;
		\addplot[blue,mark=*]  coordinates {(-1,1)};
		\addplot [red, very thick, mark=none] coordinates {(-1, 0) (-1, 1)}; 
		\addplot[blue,mark=*]  coordinates {(1,-1)};
		\addplot[blue]  coordinates {(-1.1, 1.3)} node{{$-x_1^n+x_2^n$}};
		\addplot[red, very thick,domain=0:1] plot {-1}; 
		\addplot [red, very thick, mark=none] coordinates {(1, -1) (1, 0)} ;
		\addplot[blue]  coordinates {(0.9,-1.4)} node{{$x_1^n-x_2^n$}};
		\addplot[blue]  coordinates {(-1.4,0.3)} node{{$-x_1^n$}};
		\end{axis}
		\end{tikzpicture}
	\end{center}
	
\end{minipage}


\begin{minipage}{0.46\linewidth}
	\medskip
	\begin{center}
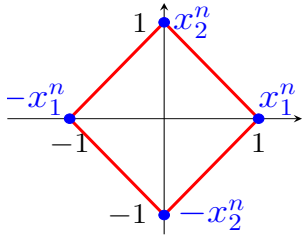
		{Unit ball of
			$\bigl(\opolyK{n},\|\cdot\|_r\bigr)$ for any $n$ \\
			and $(\opoly{n}{\CK},\|\cdot\|_\infty)$ for $n$ odd.}
	\end{center}
\end{minipage}
\hfill
\begin{minipage}{0.46\linewidth}
	\begin{center}
		{Unit ball of
			$\bigl(\opolyK{n},\|\cdot\|_\infty\bigr)$ \\ for $n$ even.}    
	\end{center}
\end{minipage}



\bigskip

\subsection{The isometries of $(C(K), \normcdot)$}
We would like next to determine the isometries
of the spaces $\opolyK{n}$, both for the 
regular and the supremum norms. 
Our results show that this reduces to the 
problem of finding the isometries between
the spaces $\MK$ for the variation norm
and the equivalent norm $\|\cdot\|_0$.

The Banach-Stone theorem \cite[V.8.8]{DS}
uses the classification of the extreme points 
of the space of regular Borel signed measures
to determine the isometries of $C(K)$ spaces
with the supremum norm.
We recall the statement of this theorem:
if $T$ is an isometric isomorphism
between $C(K)$ and $C(L)$,
then there exists a homeomorhism
$\varphi\colon L\to K$ and a
function $\alpha\in C(L)$ with
values $\pm 1$, such that
\begin{equation}\label{e:BS}
\bigl(Tx\bigr)(s)  = \alpha(s) x(\varphi(s))
\end{equation}
for all $x\in \CK$, $s\in L$.
We shall say that an linear
bijection, $T$, from $C(K)$ to $C(L)$ is 
\emph{canonical} if it has this form.
In other words, 
$$
Tx = \alpha\,x\circ\varphi \,,
$$
where $\alpha$, $\varphi$ are as described above.

Consider the space $\bigl(\MK,\|\cdot\|_0\bigr) \cong \bigl(C(K),\normcdot
\bigr)'$. 
By Theorem~\ref{p:extremeMK}, the 
set of extreme points of the unit
ball of $({\cal M}(K),\|\cdot\|_0)$ is  $\{\pm\delta_u,\delta_t-\delta_s
:u,t,s\in K, t\neq s\}$. The crucial step in showing that an isometry $T$ of 
from $(C(K),\normcdot)$ to $(C(L),\normcdot)$  is
canonical is to establish that $T^t$, the transpose of $T$, maps each
$\delta_t$ to $\pm\delta_s$ for some $s$ in $K$. This leads to the following 
proposition.

\begin{proposition}\label{morethan2}
	Let $K$ and $L$ be compact Hausdorff topological spaces and let  
	$T\colon \bigl(C(K),\normcdot\bigr) 
	\to \bigl(C(L),\normcdot\bigr)$ be an isometric isomorphism. Let 
	$S_L=\{t\in L: T^t(\delta_t)=\pm\delta_s,\mbox{ for some } s\in K\}$. If $S_L$ 
	contains more than one point, then $T$ is canonical. Moreover, in addition,
	$\alpha$ will either take the constant value $1$ or $-1$ on $L$.
\end{proposition}

\begin{proof} Assume that $|S_L|\ge 2$ and $S_L^c$
	is non-empty. Choose $r\in S_L^c$. Then we have that $T^t(\delta_r)=\delta_u-
	\delta_v$ for
	some $u$ and $v$ in $K$. Since $|S_L|\ge 2$, there are $t$ and $s$ in $L$ with 
	$t\neq s$ so that $T^t(\delta_t)=\pm\delta_w$ and $T^t(\delta_s)=\pm\delta_p$ 
	for some $w$ and $p$ in $K$ with  $w\neq p$. We now claim that 
	$\{w,p\}\not=\{u,v\}$ and so there $t$ in $L$ so that  
	$T^t(\delta_t)=\pm\delta_w$ with $w\not=u,v$. Without loss of generality suppose
	that $w=u$ and $p=v$. Then we have  $T^t(\delta_t)=\pm\delta_u$ and $T^t(
	\delta_s)=\pm\delta_v$. Hence $$
	T^t(\delta_t-\delta_s)=\pm\delta_u\mp\delta_v
	$$ and therefore $(T^t)^{-1}(\delta_u-\delta_v)=\pm (\delta_t-\delta_s)$. Since 
	$(T^t)^{-1}$ is a bijection we have  $\delta_r=\pm (\delta_t-\delta_s)$ which is 
	impossible.  Let $t$ in $L$ be such that  
	$T^t(\delta_t)=\pm\delta_w$ with $w\not=u,v$. Then $\delta_r-\delta_t$ is an 
	extreme point of the unit ball of $({\cal M}(L),\|\cdot\|_0)$. However, 
	$$
	T^t(\delta_r-\delta_t)=\delta_u-\delta_v\pm\delta_w\,.
	$$
	Since $w\not= u,v$, $\delta_u-\delta_v\pm\delta_s$, is not
	an extreme point 
	of the unit ball of $({\cal M}(K),\|\cdot\|_0)$. This is a contradiction. 
	Hence, if $|S_L|\ge 2$, then $S_L=L$.
	
	Note that $T\colon (C(K),\|\cdot\|_\infty) \to (C(L),\|\cdot
	\|_\infty)$ is an isomorphism since the norms $\|\cdot\|_\infty$ and
	 $\normcdot$ are equivalent. Further, since $S_L=L$, 
	 we have that for every $t$ in $L$ 
	there is $s$ in $K$ such that $T^t(\delta_t)=\pm\delta_s$. Hence $T^t$ maps 
	extreme points of the unit ball of $({\cal M}(L),\|\cdot\|_1)$ to the extreme 
	points of the unit ball of $({\cal M}(K),\|\cdot\|_1)$ in one to one manner. 
	Hence $T^t(B_{{\cal M}(L)})\subseteq B_{{\cal M}(K)}$ and  $(T^t)^{-1}(B_{{\cal M}(K)})
	\subseteq B_{{\cal M}(L)}$. This gives us that $T: (C(K),\|\cdot\|_\infty) \to (C(L
	),\|\cdot
	\|_\infty)$ is an isometric isomorphism. Hence we can now apply Banach-Stone 
	theorem to find a homeomorphism $\varphi$ from $L$ to $K$ and a function  
	$\alpha\in C(K)$  with $\alpha(t)=\pm 1$ for all $t\in K$ such that    
	
	$$
	T(x)=\alpha\, x\circ\varphi.
	$$
	
	Now let us see that $\alpha$ is constant on $L$. To see this suppose that 
	$$
	S_L^+=\{t\in L:T^t(\delta_t)=\delta_s \hbox{ for some } s\in K\}
	$$
	and
	$$
	S_L^-=\{t\in L:T^t(\delta_t)=-\delta_s \hbox{ for some } s\in K\}
	$$ 
	are both non empty. Choose $t$ in $S_L^+$ and $r$ in $S_L^-$. Suppose that
	$T^t(\delta_t)=\delta_u$ and that $T^t(\delta_r)=-\delta_v$. Then $\delta_t-
	\delta_r$
	is an extreme point of the unit ball of  $({\cal M}(L),\|\cdot\|_0)$ yet 
	$T^t(\delta_t-\delta_r)=\delta_u+\delta_v$ is not extreme point of the unit ball
	of $({\cal M}(K),\|\cdot\|_0)$. The result now follows and we get that
	$$
	T(x)=\pm\, x\circ\varphi.
	$$
\end{proof}

Let us now consider the case when $|S_L|=1$ and show that we can construct a 
non canonical isometry in this case. 
To help understand this result, we first consider the 
following example.

\begin{example}
	
	Let $K=\{a,b\}$ and 
	$L=\{\alpha,\beta\}$. We observe that we can identify both $(C(K),\normcdot)$ 
	and $(C(L),\normcdot)$ with $\mathbb{R}^2$. Let $x$ in $C(K)$ and set $x_1=x(a)$ and $x_2=x(b)$.
	Then $(x_1,x_1) \in \mathbb R^2$ and the norm of $(x_1,x_1)$ is given by
	$$
	\normc{(x_1,x_2)}=\max\{|x_1|,|x_2|,|x_1-x_2|\}.
	$$

	Now define $T: (\mathbb{R}^2,\normcdot) \to (\mathbb{R}^2,\normcdot)$ by
	
	$$
	T(x_1,x_2)=(x_1,x_1-x_2)
	$$
	
	Clearly, $T$ is a continuous linear bijection. We can also show that 
	\begin{align*}
	T^t(\delta_{\alpha})=&\delta_a,\\
	T^t(\delta_{\beta})=&\delta_a-\delta_b.
	\end{align*}
	
	We have that
	$$
	\normc{T(x_1,x_2)}=\max\{|x_1|,|x_1-x_2|,|x_2|\}=
	\normc{(x_1,x_2)}
	$$
	and hence $T$ is an isometry. However, $T$ is not canonical since
	
	\begin{align*}
	(Tx)(\alpha) & = x(a)\,,\\
	(Tx)(\beta) & = x(a)-x(b)\,.
		\end{align*}
	
\end{example}

Guided by Proposition \ref{morethan2} and the above example,
we now have the following result.

\begin{theorem}
	Let $K$ and $L$ be compact Hausdorff topological spaces. 
	\begin{enumerate}
		\item[(a)] Suppose that $K$ and $L$ do not contain isolated 
		points. Then every isometric isomorphism $T$ from $(C(K),\normcdot)$ onto 
		$(C(L),\normcdot)$ has the form 
		$$
		T(x)=\pm\, x\circ\varphi.
		$$
		for some homeomorphism $\varphi\colon L\to K$.
		\item[(b)] Suppose that either $K$ or $L$ contains an isolated 
		point. Let $T:(C(K),\normcdot) \to (C(L),\normcdot)$ be an isometric 
		isomorphism. Then $T$ is one of the following types.
		\begin{enumerate}
			\item[(i)] 
			$$
			T(x)=\pm\, x\circ\varphi.
			$$
			for some homeomorphism $\varphi\colon L\to K$.
			\item[(ii)] There exist $p$ in $K$ and $t$ in $L$ and a 
			homeomorphism $\varphi\colon L\setminus\{t\}\to K\setminus\{p\}
			$ such that $T=\pm T_1$, where 
			\begin{align*}
			(T_1x)(t)&=x(p)\\
			(T_1 x)(s)&=x(p)-x(\varphi(s))
			\qquad\text{for $s\neq t$.}
			\end{align*}
		\end{enumerate}
	\end{enumerate}
\end{theorem}

\begin{proof} (a) Note that $L=S_L\cup S^c_L$. 
	We claim that, if $|S_L|=1$, then $L$ contains 
	an isolated point.
	Suppose that $S_L=\{t\} $  and,
	without loss of generality,  $T^t\delta_t= \delta_s$.   Then
	$$
	(T 1_K) (t) =\delta_t(T 1_K) = (T^t\delta_t)(1_K)= 
	\delta_s(1_K) =1 \,.
	$$
	For any $r\in S_L^c$, a similar calculation shows
	that $(T 1_K)(r) = 0$.
	As $S_L= (T 1_K)^{-1}(1)$ and 
	$S_L^c= (T 1_K)^{-1}(0)$ and $T1_K$ is continuous,
	it follows that $S_L$ and $S_L^c$ are disjoint
	closed sets.  Therefore $S_L=\{t\}$ is an isolated
	point of $L$.
	Therefore, if $L$ does not contain isolated points then $|S_L|\geq 2$ 
	and Proposition~\ref{morethan2} gives us that $T$ is canonical.

	(b) We only need to consider the case $|S_L|=1$ as otherwise 
	Proposition~\ref{morethan2} gives us that $T$ is canonical. Suppose $L$ contains an 
	isolated point $t$ and $K$ an isolated point
	$p$. For each $x$ in $C(K)$ the function $Tx$ as defined in (b) is continuous
	and the mapping $x\to Tx$ is easily seen to be an isometry. 
	
	Let us see that if $T$ is not canonical then this is the form that an isometry
	can take. By definition and the fact that $T$ is invertible we have that 
	$|S_K|=1$, where $S_K$ is the set of points $q$ in $K$
	for which $T^t (\delta_t) = \pm \delta_q$ for some $t\in L$.
	Let $S_L=\{t\}$ and $S_K=\{p\}$. Then $T^t(\delta_t)=\pm\delta_p$. 
	Let us suppose that $T^t(\delta_t)=\delta_p$. We claim that for each $s$ in 
	$S_L^c$ we have $T^t(\delta_s)=\delta_p-\delta_q$ for some $q$ in $K\setminus\{p
	\}$. Otherwise we have that $\delta_t-\delta_s$ is extreme but $T^t(\delta_t-
	\delta_s)=\delta_p-\delta_u+\delta_v$ is not. The mapping $T^t(\delta_s)=
	\delta_p-\delta_q$ now induces a bijection $\varphi\colon L\setminus\{t\}\to K
	\setminus\{p\}$ so that $T^t(\delta_s)=\delta_p-\delta_{\varphi(s)}$. Since the
	mapping $L\setminus\{p\}\to ({\cal M}(K),\sigma({\cal M}(K),{\cal C}(K)))$,
	$s\mapsto \delta_p-\delta_{\varphi(s)}$, is
	continuous, $\varphi$ will be continuous. As $\varphi$ is a continuous 
	bijection from the compact  space 
	$L\setminus \{t\}$ to the Hausdorff  space
	$K\setminus \{p\}$ it
	is a homeomorphism. Rewriting $s\mapsto \delta_p -
	\delta_{\varphi(s)}$
	in terms of $x$, we see that $(Tx)(s)=x(p)-x(\varphi(s))$.
	When $T^t(\delta_t)= -\delta_s$, we obtain
	$(Tx)(s)=x(\varphi(s))-x(p)$.
\end{proof}

Our characterisation of the isometries of $(C(L),\normcdot)$ onto $(C(K),\normcdot)$ allows us to construct isometries of $(\opoly{n}{C(K)}, \|\cdot\|_\infty)$ onto $(\opoly{n}{C(L)},\|\cdot\|_\infty)$. Given a homeomorphism 
$\varphi\colon K\to L$ we use $C_\varphi$ to denote the composition operator
$C_\varphi\colon C(L)\to C(K)$ defined by $C_\varphi(f)=f\circ\varphi$ for each
$f$ in $C(L)$. The transpose of the canonical isometry of $(C(K),\normcdot)$
onto $(C(L),\normcdot)$ determined by $\varphi$
now gives rise to the isometry
$T\colon (\opoly{n}{C(K)},\|\cdot\|_\infty)
\to (\opoly{n}{C(L),\|\cdot\|_\infty})$ 
given by $T(P)=P\circ C_\varphi$.

To understand the isometries from $(\opoly{n}{C(K)}, \|\cdot\|_\infty)$ to 
$(\opoly{n}{C(L)},\|\cdot\|_\infty)$ 
induced by non canonical isometries of
$(C(L),\normcdot)$ onto $(C(K),\normcdot)$ we note that if $K$ and $L$ have
isolated points $t$ and $p$ respectively then we have that $(C(K),\|\cdot
\|_\infty)$ is isometrically isomorphic to $(C(\{t\}),\|\cdot\|_\infty)\oplus_
\infty(C(K\setminus\{t\}),\|\cdot\|_\infty)$ while $(C(L),\|\cdot\|_\infty)$ is
isometrically isomorphic to $(C(\{p\}),\|\cdot\|_\infty)\oplus_\infty(C(L
\setminus\{p\}),\|\cdot\|_\infty)$. Hence, if $P$ is an $n$-homogeneous 
orthogonally additive polynomial on $(C(K),\|\cdot\|_\infty\|)$ then we can
write $P$ as $P=\lambda \delta_t^n+P_2$ where $P_2=P|_{C(K\setminus\{t\})}$. It 
follows that the transpose of each non canonical isometry
from  $(C(K),\normcdot)$
onto $(C(L),\normcdot)$ gives an isometry from $(\opoly{n}{C(K)},\|\cdot
\|_\infty)$ onto $(\opoly{n}{C(L)},\|\cdot\|_\infty)$ of the form
$$
T(P)=P(1)\,\delta_p^n-P_2\circ C_\varphi
$$
where  $\varphi$ is a homeomorphism of $K\setminus
\{t\}$ to $L\setminus\{p\}$.

In a similar manner, we can construct canonical
isometries from 
$(\opoly{n}{C(K)},\|\cdot\|_r)$ onto
$(\opoly{n}{C(L)},\|\cdot\|_r)$.

\section{Exposed points in $\opolyK{n}$ }

In this section we shall characterise the weak${}^*$ exposed and weak${}^*$
strongly exposed point of the unit
ball of $(C(K),\normcdot)'$. We have an upper bound for this set. We know that
it is contained in the set of extreme points of the unit ball of $(C(K),
\normcdot)'\cong ({\cal M}(K),\|\cdot\|_0)$ and that the set of extreme 
points of this set is equal to $\{\pm\delta_p,\delta_t-\delta_s:p, t,s\in K, t\neq s\}$. 

Let us begin with some definitions.

\begin{definition}
	Let $E$ be a Banach space. A point $x$ in the closed unit ball of $E$ is said 
	to be an 
	\emph{exposed point}  if there exists $\varphi \in E'$ with $\|\varphi\|=
	1$ such that
	\[
	\varphi(x) = 1 \text{ and } \varphi (y) < 1 \text{ for } y \in  
	\overline{B}_E \backslash \{x\}.
	\]
	If this is the case then we say that $\varphi$ \textit{exposes} $x$.
\end{definition}

\begin{definition}
	We say that $x$ is  a \emph{strongly exposed point} of 
	the closed unit ball of $E$ if there exists 
	 $\varphi \in E'$ such that
	\[
	\varphi(x) = 1
	\]
	and whenever $(x_n)_n$ is a sequence in $\overline{B}_E$ with 
	$
	\lim_{n \rightarrow \infty} \varphi(x_n)=1
	$
	then $(x_n)_n$ converges  to $x$ in norm.
	We will say that $\varphi$ \textit{strongly exposes} $x$.
\end{definition}

If $E=F'$ is a dual Banach space and the point $x\in E$ is exposed (respectively, strongly 
exposed) by $\varphi$ in $F$
we say that $x$ is a  \emph{weak${}^*$ exposed} 
(respectively, \emph{weak${}^*$ strongly exposed}) 
point of $E$
and that $\varphi$ \emph{weak${}^*$ exposes} 
(respectively, \emph{weak$^*$-strongly exposes})
the unit ball of $E$ at $x$.

We also observe that if each $\delta_t$, $t\in K$ and each $\delta_t-\delta_s$,
$t,s\in K$ with $t\not=s$ are of norm $1$ in $(C(K),\normcdot)'$. Hence, if $\delta_t$ is 
exposed by $x$ then we must have ${\rm diam}(x)<1$. Conversely, if 
$\delta_t-\delta_s$ is exposed by $x$ then we must have $\|x\|_\infty<1$.

We note that if $K$ is a compact Hausdorff topological space then a net 
$(t_\alpha)_\alpha$ converges to $t$ in $K$ if and only if $y(t_\alpha)$ converges 
to $y(t)$ for every $y$ in $C(K)$.

\subsection{G\^ateaux differentiability of the norm} 
We start with a 
characterisation of  G\^ateaux differentiability of the norm on $(C(K),\normcdot)$.
\begin{theorem}\label{gatsingle}
	Let $K$ be a compact Hausdorff topological space. Let $t\in K$, $x\in C(K)$
	with $\normc{x}=1$. Then the following are equivalent
	\begin{enumerate}
		\item[(a)] The norm of $(C(K),\normcdot)$ is G\^ateaux differentiable at $x$
		with derivative $\delta_t$.
		\item[(b)] 
		\begin{enumerate}
			\item[(i)] $\normc{x}=x(t)=1$ and ${\rm diam}(x)<1$.
			\item[(ii)] If $(t_n)_n$ is a sequence of points in $K$ such that $\lim_{n\to 
				\infty}x(t_n)=1$ then $(t_n)_n$ has a subnet, $(t_\alpha)_\alpha$ such that $(t_
			\alpha)_\alpha$ converges to $t$.
		\end{enumerate}
		\item[(c)] $t$ is the unique point in $K$ with $x(t)=1$ and ${\rm diam}(x)<1$.
	\end{enumerate}
\end{theorem}

\begin{proof} First observe that {\v S}mul'yan \cite{Smulyan1,Smulyan2} 
	(see also \cite{DGZ}) showed that 
a point $x$ in $B_{C(K)}$ weak${}^*$ exposes the unit ball of 
$(C(K),\normcdot)'$ at $\delta_t$ if and only if the norm of $C(K)$ is
 G\^ateaux differentiable
at $x$ with derivative $\delta_t$. Hence we have that (a) implies (c).

Let us see that (c) implies (b). Clearly we have that (c) implies (b)~(i).

Suppose that (c) is true and that (b) part (ii) fails. Then there is a sequence 
$(t_n)_n$ in $K$ with $\lim_{n\to\infty}x(t_n)=x(t)=1$ but that for all subnets 
$(t_\alpha)_\alpha$ of $(t_n)_n$ there is $y$ in $C(K)$ such that $y(t_\alpha)\not
\to y(t)$. As $K$ is compact, we can choose a subnet $(t_\alpha)_\alpha$ of 
$(t_n)_n$ and $s$ in $K$ so that $\lim_{\alpha\to\infty}t_\alpha=s$. We claim that
$t\not=s$. Suppose $t=s$. Then for every $y$  
in $\CK$ we have that $\lim_\beta 
y(t_\beta)=y(t)$ contrary to what we have assumed. As $s\not=t$ and $x$ is 
continuous we have $x(s)=\lim_\beta x(t_\beta)=1$ which contradicts (c). Hence,
we see that (c) implies (b).

Next suppose that (b) is true and that (a) is false. Then we can find $y$ in 
$C(K)$, $\varepsilon>0$ and a sequence of positive numbers $(\lambda_n)_n$ 
converging to $0$ so that 
$$
\bigl|\normc{x+\lambda_ny}-\normc{x}-\lambda_n y(t)\bigr|\ge \varepsilon\lambda_n
$$
for every positive integer $n$.

Note that as $\normc{x+\lambda_n y}\ge(x+\lambda_n y)(t)$ we actually have that
$\normc{x+\lambda_ny}-\normc{x}-\lambda_n y(t)$
 is non-negative and therefore we 
have
$$
\normc{x+\lambda_ny}-\normc{x}-\lambda_n y(t)
\ge \lambda_n\varepsilon
$$
for every positive integer $n$.

Each of the functions $x+\lambda_n y$ attains 
its
norm either at a point of the form $\delta_t$ or at a point $\delta_u-\delta_v$.
As ${\rm diam}(x)<1$ choosing $n$ sufficiently large we can assume that
$x+\lambda_n y$ attains its norm at a point of the first type. Hence, for each
$n$ in $\mathbb{N}$, we can find $t_n$ in $K$ and $\beta_n=\pm 1$ so that
$$
\beta_n(x+\lambda_n y)(t_n)=\normc{x+\lambda_n y}\,.
$$
Then we have
\begin{align*}
1&= \normc{x} \ge  \beta_n x(t_n)=\beta_n(x+\lambda_n y)(t_n)-\beta_n\lambda_n y(t_n)\\
&\ge \normc{x+\lambda_n y}-|\lambda_n|\normc{y}.
\end{align*}
As $(\lambda_n)_n$ is a null sequence we have that $\normc{x+\lambda_n y}-
|\lambda_n|\normc{y}$ converges to $\normc{x}$ as $n$ tends to $\infty$. Hence we have 
that $\beta_nx(t_n)\to 1$. However, as ${\rm diam}(x)<1$, we have $x(t_n)>0$ for
all $n$. Hence, without loss of generality, we may assume that $\beta_n=1$ for 
all $n$ and therefore we have $\lim_{n\to\infty}x(t_n)=1$.

Then we have that
\begin{align*}
\varepsilon \lambda_n\le &\normc{x+\lambda_n y}-\normc{x}-\lambda_n y(t)\\
= &(x+\lambda_n y)(t_n)-x(t)-\lambda_n y(t)\\
= & x(t_n)-x(t)+\lambda_n\left( y(t_n)-y(t)\right)\\
\le & \lambda_n\left(y(t_n)-y(t)\right)\\
\end{align*}
However this means that there is no subnet
$(t_\alpha)$ of $(t_n)_n$ so that $y(t_\alpha)$ converges to $y(t)$ and so (b)~(ii)
is false.
\end{proof}

We recall that a function $x\in \CK$ is said to peak at a point
$t\in K$ if $t$ is the unique point at which $x$ attains its maximum.

\begin{lemma}\label{l:peak}
	Let $K$ be a compact Hausdorff topological space and $t\in K$. Then there is
	$x$ in $C(K)$ which peaks at $t$ if and only if $\{t\}$ is a $G_\delta$ subset 
	of $K$.
\end{lemma}

\begin{proof} We first suppose that $\{t\}$ is a $G_\delta$ subset of 
	$K$. Then we can find a sequence of open sets $(U_n)_n$ so that $\{t\}=\bigcap_
	{n=1}^\infty U_n$. As $K$ is compact and Hausdorff it is completely regular. 
	Hence, for each $n\in \mathbb{N}$ we can find a continuous function $x_n\colon 
	K\to [0,1]$ such that $x_n(t)=1$ and $x_n(U_n^c)=0$. Now let  $x\colon K\to 
	[0,1]$ be defined by 
	$$
	x(t)=\frac{6}{\pi^2}\sum_{n=1}^\infty \frac{1}{n^2}x_n(t).
	$$
	Then we have $x(t)=1$ and $x(s)<1$ for $s\in K$, $s\not=t$. So $x$ peaks at $t$.
	
	Conversely, if there is $x$ in $C(K)$ which peaks at $t$, for each $n\in \mathbb
	{N}$ let $U_n=\{s\in K:x(s)>1-\frac{1}{n}$. Then $\{t\}=\bigcap_{n=1}^\infty U_n
	$. As each $U_n$ is open, $\{t\}$ is a $G_\delta$ set.\end{proof}

The weak${}^*$ exposed points of the ball of the form $\delta_t$ are 
characterised by the following proposition.

\begin{proposition}\label{peak}
	Let $K$ be a compact Hausdorff topological space. 
	Then $\{\pm \delta_t:t\in K\}$ is contained in the set of weak${}^*$ exposed 
	points of the unit ball of $({\cal M}(K),\|\cdot\|_0)$ if and only 
	if $K$ is first countable.
\end{proposition}

Just as we have characterised the  weak${}^*$ exposed points of the ball of 
the form $\delta_t$ we now characterise weak${}^*$ exposed points of the form
$\delta_t-\delta_s$.
Replacing $\delta_t$ with $\delta_t-\delta_s$ in Theorem~\ref{gatsingle} we
obtain the following result.

\begin{theorem}
	Let $K$ be a compact Hausdorff topological space. Let $t,s\in K$, $x\in C(K)$
	with $\normc{x}=1$. Then the following are equivalent
	\begin{enumerate}
		\item[(a)] The norm of $(C(K),\normcdot)$ is G\^ateaux differentiable at $x$
		with derivative $\delta_t-\delta_s$.
		\item[(b)] 
		\begin{enumerate}
			\item[(i)] $\normc{x}=x(t)-x(s)=1$ and $\|x\|_\infty<1$.
			\item[(ii)] If $(t_n)_n$ and $(s_n)_n$ are sequences of points in $K$ such that 
			$\lim_{n\to \infty}x(t_n)-x(s_n)=1$ then $(t_n)_n$ and $(s_n)$ have subnets
			$(t_\alpha)_\alpha$ and $(s_\alpha)_\alpha$ which converge to $t$ and $s$ 
			respectively.
		\end{enumerate}
		\item[(c)] $t,s$ is the unique pair of points in $K$ with $x(t)-x(s)=1$ and 
		$\|x\|_\infty<1$.
	\end{enumerate}
\end{theorem}

As the proof  of the following lemma is similar to that of Lemma~\ref{l:peak}
we omit it.\
\begin{lemma}
	Let $K$ be a compact Hausdorff topological space and $t,s\in K$ with $t\not=s$. 
	Then there is $x$ in $C(K)$ such that $x(t)=\frac{1}{2}$, $x(s)=-\frac{1}{2}$ 
	and $-\frac{1}{2}<x(u)<\frac{1}{2}$ for $u\in K\setminus\{t,s\}$ if and only
	if $\{t\}$ and $\{s\}$ are $G_\delta$ sets.
\end{lemma}

The weak${}^*$ exposed points of the ball of the form $\delta_t-\delta_s$ are now
characterised by the following proposition.

\begin{proposition}\label{2peak}
	Let $K$ be a compact Hausdorff topological space and $n$ be an even integer. 
	Then $\{\delta_t-\delta_s:t,s
	\in K, t\neq s\}$ is contained in the set of weak${}^*$ exposed points of the unit ball 
	of $({\cal M}(K),\|\cdot\|_0)$ if and only if $K$ is first countable.
\end{proposition}

Propositions~\ref{peak} and \ref{2peak} can be rephrased in terms of 
spaces of orthogonally additive polynomials.
Since we have canonically identified the space $\opoly{n}{\CK}$   with the space $\MK$,
we may transfer the weak$^*$ topology on $\MK
= \CK'$ to the space $\opoly{n}{\CK}$.
References to the weak$^*$ topology on $\opoly{n}{\CK}$
should be understood in this sense. 
It is easy to see that this is the topology of
pointwise convergence on $\opoly{n}{\CK}$.

\begin{proposition}\label{peakpoly}
	Let $K$ be a compact Hausdorff topological space and $n$ be an even integer. 
	Then $\{\pm\delta_p^n,\delta_t^n-\delta_s^n:p, t,s\in K, t\neq s\}$
	is equal to the set of weak$^*$ exposed points of the unit ball of
	$(\opoly{n}{\CK},\|\cdot\|_\infty)$ if and only if $K$
	is first countable.
\end{proposition}

\subsection{Fr\'echet differentibility of the norm}

We now characterise Fr\'echet differentiability of the norm on $(C(K),\normcdot)$.

\begin{theorem}\label{Frechet}
	Let $K$ be a compact Hausdorff topological space. Let $t\in K$, $x\in C(K)$ 
	with $\normc{x}=1$. Then the following are equivalent.
	\begin{enumerate}
		\item[(a)] The norm of $(C(K),\normcdot)$ is Fr\'echet differentiable at $x$
		with derivative $\delta_t$.
		\item[(b)] 
		\begin{enumerate}
			\item[(i)] $\normc{x}=x(t)=1$ and ${\rm diam}(x)<1$.
			\item[(ii)] If $(t_n)_n$ is a sequence of points in $K$ such that $\lim_{n\to 
				\infty}x(t_n)=1$ then $(t_n)_n$ is eventually equal to $t$.
		\end{enumerate}
		\item[(c)]$x$ weak${}^*$ strongly exposes the unit ball of
		 $(C(K),\normcdot)'$ at $\delta_t$.
	\end{enumerate}
\end{theorem}

\begin{proof} First observe that {\v S}mul'yan \cite{Smulyan1,Smulyan2} (see also \cite{DGZ})  showed that 
a point $x$ in $B_{C(K)}$ weak${}^*$ strongly exposes the unit ball of $(C(K),
\normcdot)'$ at $\delta_t$ if and only if the norm of $C(K)$ is Fr\'echet 
differentiable at $x$ with derivative $\delta_t$. Thus (a) and (c) are 
equivalent.

If the norm of $(C(K),\normcdot)$ is Fr\'echet differentiable at $x$ with 
derivative $\delta_t$ then it is G\^ateaux differentiable at $x$ with 
derivative $\delta_t$. Theorem~\ref{gatsingle} now implies that (b)~(i) holds. 

Suppose that (c) is true. Then $x$ in $B_{C(K)}$ weak${}^*$ strongly exposes 
the unit ball of $(C(K),\normcdot)'$ at $\delta_t$. If $\lim_{n\to \infty}x(t_n)
=1$ then $\lim_{n\to\infty}\delta_{t_n}(x)=\delta_t(x)=1$. As 
$f$ weak${}^*$-strongly exposes the unit ball of $(C(K),\normcdot)'$ at 
$\delta_t$ we have that $\lim_{n\to\infty}\delta_{t_n}=\delta_t$ in norm. However, 
as 
$\|\delta_u-\delta_v\|_0=1$ whenever $u\not= v$ we see that only way we can have
$(\delta_{t_n})_n$ converge to $\delta_t$ is that the sequence $(t_n)_n$ is
eventually equal to $t$.

The implication (b) implies (a) is similar to the corresponding part of the
proof of 
Theorem~\ref{gatsingle}  where instead of using the fact that $(t_n)_n$ has a
subsequence that converges to $t$ we use the fact that $(t_n)_n$ has a
subsequence so that it is eventually equal to $t$.\end{proof} 

\begin{corollary}
	Let $K$ be a compact Hausdorff topological space and $t\in K$. Then $\delta_t$
	is a weak${}^*$ strongly exposed point of the unit ball of $(C(K),\normcdot)'$ if and only if $t$ is an isolated point of $K$.
\end{corollary}

\begin{proof} If $t$ is an isolated point of $K$ then the function 
given by
$$
x(s)=\begin{cases} 1, & s=t\\
1/2 &\hbox{otherwise}\\
\end{cases}
$$
is continuous on $K$. Moreover, if $x(t_n)\to 1$ then $(t_n)_n$ is eventually
equal to $t$. 

Conversely, if $t$ is not an isolated point of $K$. Choose a sequence of points 
$(t_n)_n$ with $t_n\not=t$, all $n$, so that $t_n$ converges to $t$. Let $x$ 
be any function in $C(K)$ with $\normc{x}=1$ and $x(t)=1$. Then we have that 
$x(t_n)\to x(t)=1$. 
However, as $(t_n)_n$ is not eventually equal to $t$ we see that condition 
(b)~(ii) of Theorem~\ref{Frechet} is 
not satisfied and therefore no $x$ in $C(K)$ with $\normc{x}=1$ can expose the 
unit ball $(C(K),\normcdot)'$ at $\delta_t$.\end{proof}

\begin{theorem}
	Let $K$ be a compact Hausdorff topological space. Let $t,s\in K$, $x\in C(K)$
	with $\normc{x}=1$. Then the following are equivalent
	\begin{enumerate}
		\item[(a)] The norm of $(C(K),\normcdot)$ is Fr\'echet differentiable to $x$
		with differential $\delta_t-\delta_s$.
		\item[(b)] 
		\begin{enumerate}
			\item[(i)] $\normc{x}=x(t)-x(s)=1$ and $\|x\|_\infty<1$.
			\item[(ii)] If $(t_n)_n$ and $(s_n)_n$ are sequences of points in $K$ such that 
			$\lim_{n\to \infty}x(t_n)-x(s_n)=1$ then $(t_{n})_n$ and $(s_{n})_n$ are eventually
			the constant sequences $t$ and $s$ respectively.
		\end{enumerate}
		\item[(c)]$x$ weak${}^*$ strongly exposes the unit ball of $(C(K),\normcdot)'$ at $\delta_t-\delta_s$.
	\end{enumerate}
\end{theorem}

\begin{proof} The proof is similar to Theorem~\ref{Frechet} and
therefore omitted.\end{proof}

\begin{corollary}
	Let $K$ be a compact Hausdorff topological space and $t,s\in K$ with $t\not=s$. 
	Then $\delta_t-\delta_s$ is a weak${}^*$ strongly exposed point of the unit 
	ball of $(C(K),\normcdot)'$ if and only if $t$ and $s$ are isolated points of 
	$K$.
\end{corollary}

We can rephrase these results in terms of spaces of orthogonally additive 
polynomials as follows.

\begin{proposition}
	Let $K$ be a compact Hausdorff topological space, let $n$ be an even integer and let $s,t$ be distinct points in $K$.
	\begin{itemize}
	\item[(a)]
	$\delta_t^n$ is a weak${}^*$ strongly exposed 
	point of the unit ball of $(\opoly{n}{\CK},\|\cdot\|_\infty)$ 
	if and only if  $t$ is an isolated point of $K$.
	\item[(b)]
	$\delta_t^n-\delta_s^n$ ($s\neq t$) is a weak${}^*$ strongly exposed point of the 
	unit ball of $(\opoly{n}{\CK},\|\cdot\|_\infty)$  
	if and only if $t$ and $s$ are isolated points of $K$. 
\end{itemize}
\end{proposition}

In particular, we see that if $K$ has no
isolated points, then the unit ball of
$\opoly{n}{\CK}$ does not contain any
weak$^*$ strongly exposed points.

\bigskip


\centerline{\bf Acknowledgements}
We thank Dirk Werner and Tony Wickstead for helpful discussions.

\bibliographystyle{amsplain}
\bibliography{OAddC(K)}

\noindent Christopher Boyd, School of Mathematics \& Statistics, University
College Dublin, \hfil\break
Belfield, Dublin 4, Ireland.\\
e-mail: Christopher.Boyd@ucd.ie

\medskip

\noindent Raymond A. Ryan, School of Mathematics, Statistics and Applied 
Mathematics, National University of Ireland Galway, Ireland.\\
e-mail: ray.ryan@nuigalway.ie

\medskip

\noindent Nina Snigireva, School of Mathematics, Statistics and Applied 
Mathematics, National University of Ireland Galway, Ireland.\\
e-mail: nina.snigireva@nuigalway.ie

\end{document}